\documentclass[11pt]{amsart}

\usepackage{latexsym,rawfonts}
\usepackage{amsfonts,amssymb}
\usepackage{amsmath,amsthm}

\usepackage[plainpages=false]{hyperref}
\usepackage{graphicx}
\usepackage{color}

\newtheorem{theorem}{Theorem}%[section]

\newtheorem{lemma}{Lemma}
\newtheorem{proposition}{Proposition}

\theoremstyle{definition}
\newtheorem{definition}{Definition}
\newtheorem{remark}{Remark}

\newtheorem{assumption}{Assumption}
\newtheorem{fact}{Fact}
\newtheorem{question}{Question}
\newcommand{\beq}{\begin{equation}}
\newcommand{\eeq}{\end{equation}}
\newcommand{\beqs}{\begin{eqnarray*}}
\newcommand{\eeqs}{\end{eqnarray*}}
\newcommand{\beqn}{\begin{eqnarray}}
\newcommand{\eeqn}{\end{eqnarray}}
\newcommand{\beqa}{\begin{array}}
\newcommand{\eeqa}{\end{array}}

\usepackage{amsmath}
\numberwithin{equation}{section}
\numberwithin{theorem}{section}
\numberwithin{lemma}{section}
\numberwithin{remark}{section}
\numberwithin{assumption}{section}
\numberwithin{definition}{section}
\numberwithin{fact}{section}
\numberwithin{question}{section}
\begin{document}

\title{Hessian equations of Krylov type on K\"ahler manifolds}

\author{Li Chen}
\address{Faculty of Mathematics and Statistics, Hubei Key Laboratory of Applied Mathematics, Hubei University, Wuhan 430062, P.R. China}
\email{chenli@@hubu.edu.cn}

%\thanks{This research was supported by funds from Natural Science Foundation of Hubei Province, China,
%No.2020CFB606, and Natural Science Foundation of China No.11971157.}

\date{}

\begin{abstract}
In this paper, we consider a Hessian equation with its structure as a combination of elementary symmetric functions
on closed K\"ahler manifolds. We provide a sufficient and necessary condition for the solvability of this equation,
which generalizes the results of Hessian equation and Hessian quotient equation. As a consequence, we can solve a
complex Monge-Amp\`ere type equation proposed by Chen in the case that one of the coefficients is negative.
The key to our argument is a clever use of the special properties
of the Hessian quotient operator $\frac{\sigma_k}{\sigma_{k-1}}$.

\end{abstract}



\maketitle
\vskip4ex

%%%%%%%%%%%%%%%%%%%%%%%%%%%%%%%%%%%
\section{Introduction}

Let $(M, \omega)$ be a closed K\"ahler manifold of complex dimension $n\geq2$
and fix a real smooth closed $(1,1)$-form $\chi_0$ on $M$. For any $C^2$
function $u: M \rightarrow \mathbb{R}$, we can obtain a new real $(1,1)$-form
\begin{eqnarray*}
\chi_u=\chi_0+\frac{\sqrt{-1}}{2}\partial\overline{\partial} u.
\end{eqnarray*}
We consider the following Hessian type equation on $(M, \omega)$
\begin{eqnarray}\label{K-eq}
\chi^{k}_{u}\wedge \omega^{n-k}=\sum_{l=0}^{k-1}\alpha_l(z)\chi_{u}^{l}\wedge \omega^{n-l}, \quad 2\leq k\leq n,
\end{eqnarray}
where $\alpha_0(z), ..., \alpha_{k-1}(z)$ are real smooth functions on $M$.
This equation includes some of the most partial differential
equations in complex geometry and analysis.

\begin{itemize}
  \item For $k=n$, $\alpha_1=...=\alpha_{k-1}=0$, \eqref{K-eq} is the complex Monge-Amp\`ere equation
  $\chi_{u}^{n}=\alpha_0(z)\omega^{n}$ which
was famously solved on compact K\"ahler manifolds by Yau in his resolution of Calabi conjecture \cite{Yau78}, and on compact Hermitian manifolds
by Tosati-Weinkove \cite{Tos10} with some earlier work by Cherrier \cite{Che87},
Hanani \cite{Han96}, Guan-Li \cite{Guan10} and Zhang \cite{Zhang10}.
  \item For $k=n$, $\alpha_0=...=\alpha_{k-2}=0$ and the constant $\alpha_{n-1}$, \eqref{K-eq} is the $(n ,n-1)$-quotient
equation $\chi_{u}^{n}=\alpha_{n-1}\chi_{u}^{n-1}\wedge \omega$ which appears in a problem proposed
by Donalson \cite{Don99} in the setting of moment maps. After some progresses made in \cite{Chen04, Wei04, Wei04D, Wei06},
it was solved by Song-Weinokove \cite{Song08} via $J$-flow.
  \item For $k=n$ and $\alpha_0, ..., \alpha_{n-1} \in \mathbb{R}$,
\eqref{K-eq} was proposed by Chen (see Conjecture/Question 4 in \cite{Chen00}) in the study of Mabuchi energy.
A special case $\chi_{u}^{n}=\alpha_{l}\chi_{u}^{l}\wedge \omega^{n-l}$ ($1<l<n$) was solved by Fang-Lai-Ma
\cite{Fang11} by generalizing the result of Song-Weinokove \cite{Song08} and non constant $\alpha_{l}$ on Hermitian manifolds
were considered by \cite{Guan12, Guan15, Sun16}. Chen's problem was solved by Collins
and Sz\'ekelyhidi \cite{Co17} for $\alpha_i\geq 0$ ($0\leq i\leq k-1$) as was conjectured by Fang-Lai-Ma.  Moreover, a generalized
equation of Chen's problem was studied by Sun \cite{Sun14-1, Sun14-2}, Pingali \cite{Pin14-1, Pin14-2, Pin16} and Phong-T\^o \cite{Ph17-1}.
  \item The Hessian equation $\chi_{u}^{k}\wedge \omega^{n-k}=\alpha_0(z)$ and the $(k ,l)$-quotient
equation $\chi_{u}^{k}\wedge \omega^{n-k}=\alpha_l(z) \chi_{u}^{l}\wedge \omega^{n-l}$ for $1\leq l<k<n$
are also the special cases of \eqref{K-eq}. The first one was solve by Dinew-Kolodziey
\cite{Din12} (combined with the estimate of Hou-Ma-Wu \cite{Hou10}) on K\"ahler manifolds and by Sun \cite{Sun17} on Hermitian manifolds (see also Zhang \cite{Zhang17}).
The second one was solved by Sz\'ekelyhidi \cite{Sze18}
for constant $\alpha_l$ and non constant $\alpha_l$ was settled by Sun \cite{Sun17}.
\end{itemize}

All previous results on the equation \eqref{K-eq} require all coefficients $\alpha_i\geq 0$ for $0\leq i\leq k-1$. Thus, it is interesting to ask:
\begin{question}
Can we solve the equation \eqref{K-eq} for some $\alpha_i(z)$ which change sign or are negative? In particular,
can we solve Chen's problem for some $\alpha_i<0$?
\end{question}

In this paper, we make some progress on this direction. To statement our main theorem, we first need to introduce the cone condition.
\begin{definition}
Following \cite{Song08, Fang11, Guan15, Sun17}, we set
\begin{eqnarray}\label{cone}
\mathcal{C}_k(\omega):&=&\Big\{[\chi]: \exists \ \overline{\chi}
\in [\chi]\cap \Gamma_{k-1}(M), \quad \mbox{such that} \\ \nonumber&&\
k\overline{\chi}^{k-1}\wedge \omega^{n-k}>\sum_{l=1}^{k-1}l\alpha_l \overline{\chi}^{l-1}\wedge \omega^{n-l}
\Big\}.
\end{eqnarray}
where $\alpha_i(z)$ with $0\leq i\leq k-1$ are real smooth functions on $M$ and
$[\chi]=\{\chi+\frac{\sqrt{-1}}{2}\partial\overline{\partial} u: u\in C^2(M)\}$.
If $[\chi] \in \mathcal{C}_k(\omega)$, we say that $\chi$ satisfies the cone condition for \eqref{K-eq}.
\end{definition}

Next, we make the following assumption on the coefficient functions $\alpha_{i}(z)$ ($0\leq i\leq k-1$).

\begin{assumption}\label{Ass}
Suppose that $\alpha_0(z), \alpha_1(z), ..., \alpha_{k-1}(z)$ are real smooth functions on $M$
satisfying

(i) either $\alpha_l(z)>0$ or $\alpha_l(z)\equiv0$ for $0\leq l\leq k-2$ and all $z \in M$;

(ii) $\sum_{l=0}^{k-2}\alpha_l(z)>0$ for all $z \in M$;

(iii) $\alpha_{l}(z)\geq c_{k,l}$ for $0\leq l\leq k-1$ and
\begin{eqnarray}\label{Nece}
\int_{M}\chi_{0}^{k}\wedge \omega^{n-k}\leq \sum_{l=0}^{k-1}c_{k,l}\int_{M}\chi_{0}^{l}\wedge \omega^{n-l},
\end{eqnarray}
where $c_{k, 0}, ..., c_{k, k-1}$ are constants.
\end{assumption}

\begin{remark}
We make no sign requirement for $\alpha_{k-1}(z)$ in Assumption \ref{Ass}.
\end{remark}

Based on Lemma \ref{k-i}, it is easy to see the condition $[\chi_0] \in \mathcal{C}_k(\omega)$
is necessary for the solvability of \eqref{K-eq}. We show such condition is also sufficient.

\begin{theorem}\label{Main}
Let $(M, \omega)$ be a closed K\"ahler manifold of complex dimension $n\geq2$ and $\chi_0$ be a
smooth closed real $(1,1)$-form on $M$.
Assume that $\alpha_0(z), \alpha_1(z), ..., \alpha_{k-1}(z)$
satisfies Assumption \ref{Ass}.
Then, there exist a unique (up to a constant) smooth function $u$ and $a \in \mathbb{R}$ satisfies
\begin{eqnarray*}
\chi^{k}_{u}\wedge \omega^{n-k}
=\sum_{l=0}^{k-2}\alpha_l(z)\chi_{u}^{l}\wedge \omega^{n-l}+[\alpha_{k-1}(z)+a]\chi_{u}^{k-1}\wedge \omega^{n-k+1}
\end{eqnarray*}
provided that $\chi_0 \in \mathcal{C}_k(\omega)$.
\end{theorem}

In particular, if $\alpha_0, \alpha_1, ..., \alpha_{k-1}$ are constants, the equation \eqref{K-eq} was proposed by
Chen (see Conjecture/Question 4 in \cite{Chen00}) for $k=n$ and it was solved by Collins and Sz\'ekelyhidi \cite{Co17} for $\alpha_i\geq 0$
($0\leq i\leq n-1$). As a corollary of Theorem \ref{Main},
we can solve Chen's problem for $\alpha_{k-1}<0$.

\begin{theorem}\label{Main1}
Let $(M, \omega)$ be a closed K\"ahler manifold of complex dimension $n\geq2$ and $\chi_0$ be a
smooth closed real $(1,1)$-form on $M$.
Assume $\alpha_{k-1} \in \mathbb{R}$ and $\alpha_0, \alpha_1, ..., \alpha_{k-2}$ are nonnegative real numbers
satisfying $\sum_{i=0}^{k-2}\alpha_i>0$
and
\begin{eqnarray}\label{Nece-1}
\int_{M}\chi_{0}^{k}\wedge \omega^{n-k}=\sum_{l=0}^{k-2}\alpha_l\int_{M}\chi_{0}^{l}\wedge \omega^{n-l},
\end{eqnarray}
Then,
there exist a unique (up to a constant) smooth function $u$  satisfies \eqref{K-eq} if and only
if $[\chi_0] \in \mathcal{C}_k(\omega)$.
\end{theorem}

\begin{remark}
Integrating both sides of \eqref{K-eq} on $M$, it is easy to see the equality
\eqref{Nece-1} is a necessary condition for the solvability of \eqref{K-eq} for $\alpha_i \in \mathbb{R}$ ($0\leq i\leq k-1$).
\end{remark}

The equation \eqref{K-eq} is just a Hessian type of equation with its structure
as a combination of elementary symmetric functions (see Section 2.1 for details):
\begin{eqnarray}\label{K-eq-H}
\sigma_k(\chi_u)=\sum_{l=0}^{k-1}\beta_l(z)\sigma_l(\chi_u),
\end{eqnarray}
where $\beta_l(z)=\frac{C_{n}^{k}}{C_{n}^{l}}\alpha_l(z)>0$.
Such type of equation has been firstly considered by Krylov about twenty years ago. In \cite{Kry95}, he considered the Dirichlet
problem of the following degenerate equation in a $(k-1)$-convex
domain $D$ in $\mathbb{R}^n$,
\[\sigma_k(D^2u) =
\sum_{l=0}^{k-1} \beta_l(x)\sigma_l(D^2u)\] with all coefficient
$\beta_{l}(x)\geq 0$ for $0\leq l\leq k-1$. Recently, Kryolv's equation was extended by Guan-Zhang
\cite{GZ19} to the case without the sign requirement for the coefficient
function $\beta_{k-1}(x)$ and the corresponding Neumann problems of Krylov's equation
in real and complex spaces were studied by the author with his collaborators in \cite{CCX19, CCX21}.
Moreover, such type of equations with its structure as a combination of elementary symmetric functions
arise naturally
from many important geometric problems, such as the so-called Fu-Yau equation
arising from the study of the Hull-Strominger system in theoretical physics, see Fu-Yau \cite{Fu07, Fu08}
and Phong-Picard-Zhang \cite{Ph17, Ph19, Ph20}. Furthermore, the special
Lagrangian equations introduced by Harvey and
Lawson \cite{HL} can also be written as the alternative combinations of elementary symmetric
functions. Lastly, the Krylov type of equations in prescribed curvature problems
were also studied by the author with his collaborators in \cite{Chen19, Chen20}.

Compared with Hessian equations or Hessian quotient equations, the ellipticity and concavity of
\eqref{K-eq-H} is not so easy to see. Moreover, we even have several ways to make \eqref{K-eq-H}
elliptic and concave.
In \cite{Co17, Sun14-1}, the authors rewrite the equation \eqref{K-eq-H} as the form
\begin{eqnarray}\label{K-eq-H-2}
-\sum_{l=0}^{k-1}\beta_l(z)\frac{\sigma_l(\chi_u)}{\sigma_{k}(\chi_u)}=-1.
\end{eqnarray}
The disadvantage of the form \eqref{K-eq-H-2} is that the ellipticity and concavity of
\eqref{K-eq-H-2} require all $\alpha_{i}\geq 0$ for $0\leq i\leq k-1$. In this paper, we
follow the viewpoint of Guan-Zhang \cite{GZ19} to rewrite the equation \eqref{K-eq-H} as the form
\begin{eqnarray}\label{K-eq-H-1}
\frac{\sigma_k(\chi_u)}{\sigma_{k-1}(\chi_u)}
-\sum_{l=0}^{k-2}\beta_l(z)\frac{\sigma_l(\chi_u)}{\sigma_{k-1}(\chi_u)}=\beta_{k-1}(z).
\end{eqnarray}
Then, we have the following very important fact (see Lemma \ref{lem2.5} for details) which was first observed by Guan-Zhang \cite{GZ19}
for Krylov' equation.
\begin{fact}\label{fact}
Assume $\beta_l(z)\geq 0$ for $0 \leq l \leq k-2$ and $\chi_u\in \Gamma_{k-1}(M)$, then the equation
\eqref{K-eq-H-2} is elliptic and concave. Compared with the form \eqref{K-eq-H-2}, its advantages lie in two aspects:
(i) we make no sign requirement for $\beta_{k-1}(z)$ and (ii)
the proper admissible set of solutions is $\Gamma_{k-1}$ which is larger than $\Gamma_{k}$.
\end{fact}

Fact \ref{fact} is based on the special properties of the Hessian quotient operator $\frac{\sigma_k}{\sigma_{k-1}}$
which are key to our argument.
In details, Propositions \ref{Ell} and \ref{Con}, and Lemma \ref{k-i} declare that if
$l=k-1$, some properties of Hessian quotient $\frac{\sigma_k}{\sigma_{l}}$ hold true in a larger cone $\Gamma_{k-1}$.
The root of these special properties goes back to the Newton inequality \eqref{N-1}.

In view of Fact \ref{fact}, using the equation \eqref{K-eq-H-1},
we first prove a $C^0$ estimate, generalizing the approach of \cite{Sun17}
in the case of Hessian quotient equations. On the one hand, we improve the inequalities
(2.11) and (2.12) in \cite{Sun17} to get the inequalities \eqref{C0-le-3}
and \eqref{C0-le-4} such that they do not depend on the lower bound of $\alpha_i$ ($0\leq i\leq k-1)$.
This improvement makes sure Theorem \ref{Main} even holds true for $\alpha_{k-1}$ with indefinite sign
and some $\alpha_i(z)\equiv 0$ ($0\leq i\leq k-2$).
On the other hand, in order to control the integrals of Hessian operators with lower degree in Lemma \ref{C0-00},
we need to get a new
inequality \eqref{C0-le-5} by iterating the inequality (3.15) in \cite{Sun17}.

Then, we follow
a technique of Hou-Ma-Wu \cite{Hou10} in the case of Hessian equation (which is in turn based
on ideas developed for the real Hessian equation by Chou-Wang \cite{Chou01}) to derive a second derivative bound of the form
\begin{eqnarray}\label{HMW}
\sup_{M}|\partial \overline{\partial} u|\leq C(\sup_{M}|\nabla u|^2+1).
\end{eqnarray}
To achieve such estimate, the main
task is to control a third order derivative term of the form
\begin{eqnarray}\label{F1}
\frac{F^{i\overline{i}}|\chi_{1\overline{1};i}|^2}{\chi_{1\overline{1}}^2}
\end{eqnarray}
in the case of Hessian equations or Hessian quotient equations. This term can be controlled with the help of the positive term
\begin{eqnarray}\label{F2}
-F^{i\overline{j},r\overline{s}}\chi_{i\overline{j};
1}\chi_{r\overline{s}; \overline{1}}.
\end{eqnarray}
However, in the case of the equation \eqref{K-eq-H-1}, there are an extra third order derivative term of the form
\begin{eqnarray}\label{Extra}
\sum_{l=0}^{k-2}\Big(\nabla_{\overline{1}}\beta_{l}\cdot F_{l}^{i\overline{j}}\chi_{i\overline{j};
1}+\nabla_{1}\beta_{l} \cdot F_{l}^{i\overline{j}}\chi_{i\overline{j};
\overline{1}}\Big).
\end{eqnarray}
Our first contribution is to extract a part of the term \eqref{F2} to cancel the term \eqref{Extra} and meanwhile make sure
the rest of the term \eqref{F2} is enough to help us control the \eqref{F1}. Thus, the extra third order derivative
term makes our argument significantly more delicate. Our another contribution is the inequalities
in Lemma \ref{lem2.6} and Lemma \ref{lem2.7} which are very important for our $C^2$ estimate, while similar equalities
for the case of Hessian equations or Hessian quotient equations are trivial.

A $C^1$-bound is
derived by combing a second derivative bound \eqref{HMW}, with a blow-up argument
and Liouville-type theorem due to Dinew-Kolodziej \cite{Din12}. The gradient bound combined with \eqref{HMW}
then bound $|\partial \overline{\partial} u|$.

Once we establish the a prior estimates for solutions to \eqref{K-eq}, we can use the continuity method to
prove Theorem \ref{Main}. Although the method is very standard in the study of elliptic PDEs,
it is not easy to carry out this method for \eqref{K-eq} on closed complex manifold. On the one hand,
since the essential cone condition depends on $\alpha_0, ..., \alpha_{k-1}$, we can not obtain
a priori estimates for the whole family of solutions in the method of continuity.
To overcome this difficulty, we use the idea in \cite{Sun16, Sun17} to technically set up an intermediate step
and apply the method twice. On the other hand, it is very hard to make the inverse function theorem work.
To resolve the issue, we further develop the approach in \cite{Tos10}. In this method,
a new Hermitian metric is constructed, which is a crucial new technique in implementing the method of continuity.

\section{Preliminaries}

In this section, we give some basic properties of elementary symmetric functions, which could be found in
\cite{L96, S05}, and establish some key lemmas.

\subsection{Basic properties of elementary symmetric functions}

For $\lambda=(\lambda_1, ... , \lambda_n)\in \mathbb{R}^n$,
the $k$-th elementary symmetric function is defined
by
\begin{equation*}
\sigma_k(\lambda) = \sum _{1 \le i_1 < i_2 <\cdots<i_k\leq n}\lambda_{i_1}\lambda_{i_2}\cdots\lambda_{i_k}.
\end{equation*}
We also set $\sigma_0=1$ and denote by $\sigma_k(\lambda \left| i \right.)$ the $k$-th symmetric
function with $\lambda_i = 0$.
Recall that the G{\aa}rding's cone is defined as
\begin{equation*}
\Gamma_k=\{ \lambda  \in \mathbb{R}^n :\sigma _i (\lambda ) > 0, \ \forall \ 1 \le i \le k\}.
\end{equation*}

\begin{proposition}\label{prop2.1}
Let $\lambda=(\lambda_1,\dots,\lambda_n)\in\mathbb{R}^n$ and $k
= 1, \cdots,n$, then
\begin{eqnarray}\label{2.1-1}
\sigma_k(\lambda)=\sigma_k(\lambda|i)+\lambda_i\sigma_{k-1}(\lambda|i),
\quad \forall \,1\leq i\leq n,
\end{eqnarray}
\begin{eqnarray*}\label{2.1-2}
\sum_{i=1}^{n} \lambda_i\sigma_{k-1}(\lambda|i)=k\sigma_{k}(\lambda),
\end{eqnarray*}
\begin{eqnarray*}\label{2.1-3}
\sum_{i=1}^{n}\sigma_{k}(\lambda|i)=(n-k)\sigma_{k}(\lambda).
\end{eqnarray*}
\end{proposition}

The Newton inequality and the generalized Newton-MacLaurin inequality are as follows, which will be used later.

\begin{proposition}\label{prop2.4-0}
For $1\leq k\leq n-1$ and $\lambda \in \mathbb{R}^n$, we have
\begin{eqnarray}\label{N-1}
(n-k+1)(k+1)\sigma_{k-1}(\lambda)\sigma_{k+1}(\lambda)\leq k(n-k)\sigma^{2}_{k}(\lambda).
\end{eqnarray}
In particular, we have
\begin{eqnarray}\label{N}
\sigma_{k-1}(\lambda)\sigma_{k+1}(\lambda)\leq \sigma^{2}_{k}(\lambda).
\end{eqnarray}
\end{proposition}

\begin{proposition}\label{prop2.4}
For $\lambda \in \Gamma_k$ and $n\geq k>l\geq 0$, $ r>s\geq 0$, $k\geq r$, $l\geq s$, we have
\begin{align}\label{NM}
\Bigg[\frac{{\sigma _k(\lambda )}/{C_n^k }}{{\sigma _l (\lambda )}/{C_n^l }}\Bigg]^{\frac{1}{k-l}}
\le \Bigg[\frac{{\sigma _r (\lambda )}/{C_n^r }}{{\sigma _s (\lambda )}/{C_n^s }}\Bigg]^{\frac{1}{r-s}}.
\end{align}
\end{proposition}

\begin{proposition}\label{Ell}
For $\lambda\in\Gamma_{k}$, we have if $0\leq l<k\leq n$
\begin{eqnarray*}\label{0-1-sum}
\frac{\partial[\frac{\sigma_{k}}{\sigma_{l}}]
(\lambda)}{\partial\lambda_{i}}>0, \quad \forall \ 1\leq i\leq n.
\end{eqnarray*}
But,
\begin{eqnarray*}\label{0-1-sum-1}
\frac{\partial[\frac{\sigma_{k}}{\sigma_{k-1}}]
(\lambda)}{\partial\lambda_{i}} \geq0, \quad \forall \ 1\leq i\leq
n,
\end{eqnarray*}
holds for $\lambda\in\Gamma_{k-1}$.
\end{proposition}

\begin{proposition}\label{Con}
For any $n\geq k>l\ge 0$, we have
\begin{eqnarray*}
\bigg[\frac{\sigma_{k}(\lambda)}{\sigma_{l}(\lambda)}\bigg]^{\frac{1}{k-l}}
\end{eqnarray*}
is a concave function in $\Gamma_k$. But, for any $k\geq 1$, we have
$$\frac{\sigma_{k}(\lambda)}{\sigma_{k-1}(\lambda)}$$
is a concave function in $\Gamma_{k-1}$.
\end{proposition}

We recall the G{\aa}rding's inequality (see \cite{CNS85}).

\begin{lemma}
If $\lambda, \mu \in \Gamma_k$, then for any $k \in \{1, 2, ..., n\}$
\begin{eqnarray*}
\sum_{i=1}^{n}\mu_i\sigma_{k-1}(\lambda|i)\geq k[\sigma_{k}(\mu)]^{\frac{1}{k}}[\sigma_{k}(\lambda)]^{1-\frac{1}{k}}.
\end{eqnarray*}
In particular,
\begin{eqnarray}\label{G-i}
\sum_{i=1}^{n}\mu_i\sigma_{k-1}(\lambda|i)>0.
\end{eqnarray}
\end{lemma}

\begin{lemma}\label{k-i}
If $\lambda \in \Gamma_k$, then for any $i \in \{1, 2, ..., n\}$ and $1\leq l<k\leq n$ we have
\begin{eqnarray}\label{initial}
\frac{\sigma_{k-1}(\lambda|i)}{\sigma_{l-1}(\lambda|i)}>\frac{\sigma_{k}(\lambda)}{\sigma_{l}(\lambda)}.
\end{eqnarray}
In particular, if $\lambda \in \Gamma_{k-1}$, we have
\begin{eqnarray}\label{initial-1}
\frac{\sigma_{k-1}(\lambda|i)}{\sigma_{k-2}(\lambda|i)}\geq\frac{\sigma_{k}(\lambda)}{\sigma_{k-1}(\lambda)}.
\end{eqnarray}
\end{lemma}

\begin{proof}
We get from \eqref{2.1-1}
\begin{eqnarray*}
h(\lambda_i):=\frac{\sigma_{k}(\lambda)}{\sigma_{l}(\lambda)}=
\frac{\lambda_i\sigma_{k-1}(\lambda|i)+\sigma_{k}(\lambda|i)}{\lambda_i\sigma_{l-1}(\lambda|i)+\sigma_{l}(\lambda|i)}.
\end{eqnarray*}
A direct computation implies
\begin{eqnarray*}
\frac{d}{d \lambda_i}h(\lambda_i)=\frac{\sigma_{k-1}(\lambda|i)\sigma_{l}(\lambda|i)
-\sigma_{k}(\lambda|i)\sigma_{l-1}(\lambda|i)}{\sigma^{2}_{l}(\lambda)}.
\end{eqnarray*}
If $\sigma_{k}(\lambda|i)\leq 0$, it is easy to see $\frac{d}{d \lambda_i}h(\lambda_i)>0$.
Otherwise, $\sigma_{k}(\lambda|i)>0$. Thus, $\lambda$ with $\lambda_i=0$ belongs
to $\Gamma_k$. So, using the generalized Newton-MacLaurin inequality \eqref{NM}, we have
\begin{eqnarray*}
\frac{{\sigma _k(\lambda|i)}}{{\sigma_{k-1}(\lambda|i)}}<\frac{k(n-l)}{l(n-k)}\frac{{\sigma _k(\lambda|i)}}{{\sigma_{k-1}(\lambda|i)}}
\leq\frac{{\sigma _l(\lambda|i)}}{{\sigma_{l-1}(\lambda|i)}}.
\end{eqnarray*}
Thus, $\frac{d}{d \lambda_i}h(\lambda_i)>0$. Therefore,
\begin{eqnarray*}
\frac{\sigma_{k}(\lambda)}{\sigma_{l}(\lambda)}<\lim_{\lambda_i\rightarrow +\infty}h(\lambda_i)=\frac{\sigma_{k-1}(\lambda|i)}{\sigma_{l-1}(\lambda|i)}.
\end{eqnarray*}
Hence, we obtain \eqref{initial}.
For the special case $l=k-1$, we have $\frac{d}{d \lambda_i}h(\lambda_i)\geq0$ by the Newton inequality \eqref{N}.
Thus, the inequality \eqref{initial-1} follows as before.
\end{proof}

At last, we list the following well-known result (See \cite{And94}).

\begin{lemma}
If $A=(a_{ij})$ is a Hermitian matrix, $\lambda_i=\lambda_i(A)$
is one of its eigenvalues ($i = 1, ..., n$) and
$F=F(A)=f(\lambda(A))$ is a symmetric function of $\lambda_1, ...,
\lambda_n$, then for any Hermitian matrix $B= (b_{ij})$, we
have the following formulas:
\begin{eqnarray}\label{f-2}
\frac{\partial^2 F}{\partial a_{ij}\partial
a_{st}}b_{ij}b_{st}=\frac{\partial^2 f}{\partial\lambda_p
\partial\lambda_q}b_{pp}b_{qq}+2\sum_{p<q}\frac{\frac{\partial
f}{\partial \lambda_p}-\frac{\partial f}{\partial
\lambda_q}}{\lambda_p-\lambda_q}b^{2}_{pq}.
\end{eqnarray}
In addition, if $f$ is a concave function and
\begin{eqnarray*}
\lambda_1\leq \lambda_2\leq ... \leq\lambda_n,
\end{eqnarray*}
then we have
\begin{eqnarray}\label{f-3}
f_1\geq f_2\geq ... \geq f_n,
\end{eqnarray}
where $f_i=\frac{\partial f}{\partial \lambda_i}.$
\end{lemma}

Let $\lambda(a_{i\overline{j}})$ denote the eigenvalues of a Hermitian matrix $(a_{i\overline{j}})$.
Define $\sigma_k(a_{i\overline{j}})=\sigma_k(\lambda(a_{i\overline{j}}))$.
This definition can be naturally extended to K\"ahler manifolds (and more generally, Hermitian manifolds).
Let $\mathcal{A}^{1,1}(M)$ be
the space of real smooth $(1, 1)$-forms on $(M, \omega)$.
For any $\chi \in \mathcal{A}^{1,1}(M)$, in a local normal coordinate
system of $M$ with respect with $\omega$, we have
\begin{equation*}
\chi=\frac{\sqrt{-1}}{2}\chi_{i\overline{j}}dz^i\wedge d\overline{z}^{j}.
\end{equation*}

\begin{definition}
We define $\sigma_k(\chi)$ with
respect to $\omega$ as
\begin{equation*}
\sigma_k(\chi)=\sigma_k(\chi_{i\overline{j}})=\sigma_k(\lambda(\chi_{i\overline{j}})).
\end{equation*}
The definition of $\sigma_k$ is independent of the choice of local normal coordinate
system. In fact, $\sigma_k$ can be
defined without the use of local normal coordinate by
\begin{equation*}
\sigma_{k}(\chi)=C_{n}^{k}\frac{\chi^k\wedge \omega^{n-k}}{\omega^n},
\end{equation*}
where $C_{n}^{k}=\frac{n!}{(n-k)!k!}$.
We also define the G{\aa}rding's cone on $M$ by
\begin{equation*}
\Gamma_k(M)=\{ \chi  \in \mathcal{A}^{1,1}(M):\sigma _i (\chi)>0, \ \forall \ 1 \le i \le k\}.
\end{equation*}
\end{definition}

Using the above notation, we can rewrite equation \eqref{K-eq} as the following local form:
\begin{eqnarray}\label{K-eq1}
\sigma_k(\chi_{u})
=\sum_{l=0}^{k-1}\beta_l(z) \sigma_{l}(\chi_{u}),
\end{eqnarray}
where $\beta_l(z)=\frac{C_{n}^{k}}{C_{n}^{l}}\alpha_l(z)$.

Moreover, we also have a local version of the cone condition \eqref{cone}.
\begin{lemma}\label{cone-local}
$\chi \in \mathcal{C}_k(\omega)$ is equivalent to
\begin{eqnarray}\label{cone-1}
\sigma_{k-1}(\chi|i)
-\sum_{l=1}^{k-1}\beta_l\sigma_{l-1}(\chi|i)>0
\end{eqnarray}
for any $i \in \{1, 2, ..., n\}$, where $(\chi|i)$ denotes the matrix obtained by deleting the
$i$-th row and $i$-th column of $\chi$, and $\beta_l(z)=\frac{C_{n}^{k}}{C_{n}^{l}}\alpha_l(z)$
is denoted as before.
\end{lemma}

\begin{proof}
The equality \eqref{cone-1} follows directly if we
observe that coefficient of $(n-1, n-1)$ form
$\prod_{j=1, j\neq i}^{n}dz^{j}d\overline{z}^{j}$
in $\chi^{l-1}\wedge \omega^{n-l}$ is
\begin{eqnarray*}
(l-1)!(n-l)!\sigma_{l-1}(\chi|i)=\frac{1}{l}\frac{n!}{C_{n}^{l}}\sigma_{l-1}(\chi|i).
\end{eqnarray*}
\end{proof}

\subsection{The ellipticity and concavity}

To make the equation \eqref{K-eq1} elliptic and concave,
we follow an important observation by Guan-Zhang \cite{GZ19} to rewrite \eqref{K-eq1} as the form:
\begin{eqnarray}\label{K-eq2}
F(u_{i\overline{j}}, z):=\frac{\sigma_k(\chi_{u})}{\sigma_{k-1}(\chi_{u})}
-\sum_{l=0}^{k-1}\beta_l(z)\frac{ \sigma_{l}(\chi_{u})}{\sigma_{k-1}(\chi_{u})}=\beta_{k-1}(z).
\end{eqnarray}

Then, as a corollary of Propositions \ref{Ell} and \ref{Con}, we
obtain (see also Proposition 2.2 in \cite{GZ19}):

\begin{lemma}\label{lem2.5}
If $u$ is a $C^2$ function with $\chi_u\in \Gamma_{k-1}(M)$, and $\beta_l(z)$
($0 \leq l \leq k-2$) are nonnegative, then the operator $F$ is elliptic and concave. Moreover,
if $\sum_{l=0}^{k-2}\beta_l(z)>0$ for all $z \in M$, the operator $F$ is strictly elliptic.
\end{lemma}

\subsection{Some key lemmas}

In this subsection, we prove some inequalities and lemmas which will play an important
role in the establishment of the a priori estimates. In the following,
for $\lambda(z)=(\lambda_1(z), ..., \lambda_n(z)) \in C^{0}(M, \mathbb{R}^n)$, we set
\begin{eqnarray}\label{f}
f(\lambda(z), z):=\frac{\sigma_k(\lambda(z))}{\sigma_{k-1}(\lambda(z))}
-\sum_{l=0}^{k-1}\beta_l(z)\frac{ \sigma_{l}(\lambda(z))}{\sigma_{k-1}(\lambda(z))}=\beta_{k-1}(z)
\end{eqnarray}
and denote by $f_i(\lambda)=\frac{\partial f}{\partial \lambda_i}(\lambda).$

\begin{lemma}
Let $\beta_0(z), ..., \beta_{k-2}(z)$ be nonnegative functions on $M$.
Assume that $\lambda=(\lambda_1, ..., \lambda_n) \in \Gamma_{k-1}$ and $\mu=(\mu_1, ..., \mu_n) \in \Gamma_{k-1}$,
then we have
\begin{eqnarray}\label{uii}
\sum_{i}f_i(\lambda)\mu_i\geq f(\mu)+(k-l)\sum_{l=0}^{k-1} \beta_l\frac{\sigma_l(\lambda)}{\sigma_{k-1}(\lambda)}.
\end{eqnarray}
\end{lemma}

\begin{proof}
Since $f$ is concave in $\Gamma_{k-1}$ (see Lemma \ref{lem2.5}), we have
\begin{eqnarray*}
f(\mu)\leq \sum_{i}f_i(\lambda)(\mu_i-\lambda_i)+f(\lambda),
\end{eqnarray*}
which implies
\begin{eqnarray*}
\sum_{i}f_i(\lambda)\mu_i\geq f(\mu)+(k-l)\sum_{l=0}^{k-1}
\beta_l\frac{\sigma_l(\lambda)}{\sigma_{k-1}(\lambda)},
\end{eqnarray*}
where we used the fact
\begin{eqnarray*}
\sum_{i}f_i(\lambda)\lambda_i
=f(\lambda)+\sum_{l=0}^{k-1}(k-l)\beta_l\frac{\sigma_l(\lambda)}{\sigma_{k-1}(\lambda)}.
\end{eqnarray*}
So, the proof is complete.
\end{proof}

\begin{lemma}\label{C2-212}
Let $\beta_0(z), ..., \beta_{k-2}(z)$ be nonnegative continuous functions on
$M$ and $\beta_{k-1}(z)$ be continuous functions on $M$.
Assume $\mu(z)=(\mu_1(z), ..., \mu_n(z)) \in \Gamma_{k-1}$ for any $z \in M$ and $\mu$ satisfy
\begin{eqnarray}\label{cone-2}
\frac{\sigma_{k-1}(\mu|i)}{\sigma_{k-2}(\mu|i)}
-\sum_{l=1}^{k-2}\beta_l(z)\frac{\sigma_{l-1}(\mu|i)}{\sigma_{k-2}(\mu|i)}>\beta_{k-1}(z).
\end{eqnarray}
for arbitrary $z \in M$  and $i \in \{1, 2, ..., n\}$.
Then there are constants $N, \theta>0$ depending on $|\mu|_{C^{0}(M)}$ and $|\beta_i|_{C^{0}(M)}$ with $0\leq i\leq k-1$
such that for any $z \in M$, if
\begin{eqnarray*}
\max_{1\leq i\leq n}\{\lambda_i(z)\}\geq N,
\end{eqnarray*}
we have at $z$
\begin{eqnarray}\label{C2-2-1}
\sum_{i}f_i(\lambda) (\lambda_i-\mu_i)\leq -\theta-\theta \sum_{i}f_i(\lambda)
\end{eqnarray}
or
\begin{eqnarray}\label{C2-2-2}
f_1(\lambda)\lambda_1\geq \theta.
\end{eqnarray}
\end{lemma}

\begin{proof}
For any $z \in M$, without loss of generality, we may assume $\lambda_1(z)=\max_{1\leq i\leq n}\{\lambda_i(z)\}$.
If $\lambda_1\gg -\mu_1+\epsilon$, we have
\begin{eqnarray}\label{C2-le1-1}
&&\sum_{i}f_i(\lambda) (\lambda_i-\mu_i)\nonumber\\&=&\sum_{i\geq 2}f_i(\lambda) (\lambda_i-(\mu_i-\epsilon))
-\epsilon\sum_{i}f_i(\lambda)+f_1(\lambda)(\lambda_1-\mu_1+\epsilon)\nonumber\\&\leq&\sum_{i\geq 2}f_i(\lambda) (\lambda_i-(\mu_i-\epsilon))
-\epsilon\sum_{i}f_i(\lambda)+2 f_1(\lambda)\lambda_1.
\end{eqnarray}
Choosing $\epsilon$ small enough  which is independent of $z$ such that $(\mu_1-\epsilon, ..., \mu_n-\epsilon) \in \Gamma_{k-1}$ satisfies \eqref{cone-2}.
So we can choose $\lambda_1$ large enough  which is independent of $z$ such that
$\widetilde{\mu}=(\lambda_1, \mu_2-\epsilon, ..., \mu_n-\epsilon) \in \Gamma_{k-1}$
and
\begin{eqnarray}\label{C2-le1-11}
f(\widetilde{\mu})\geq \beta_{k-1}+\widetilde{\epsilon}
\end{eqnarray}
in view that
\begin{eqnarray*}
\lim_{\lambda_1\rightarrow+\infty}f(\widetilde{\mu})=\frac{\sigma_{k-1}(\mu^{\prime})}{\sigma_{k-2}(\mu^{\prime})}
-\sum_{l=1}^{k-2}\beta_l\frac{\sigma_{l-1}(\mu^{\prime})}{\sigma_{k-2}(\mu^{\prime})}>\beta_{k-1}.
\end{eqnarray*}
where $\mu^{\prime}=(\mu_2-\epsilon, ..., \mu_n-\epsilon)$ and we use the inequality \eqref{cone-2}.
Now, we rewrite \eqref{C2-le1-1} as
\begin{eqnarray}\label{C2-le1-2}
&&\sum_{i}f_i(\lambda) (\lambda_i-\mu_i)\nonumber\\&\leq&\sum_{i}f_i(\lambda)(\lambda_i-\widetilde{\mu}_i)
-\epsilon\sum_{i}f_i(\lambda)+2 f_1(\lambda)\lambda_1.
\end{eqnarray}
Since $f$ is concave in $\Gamma_{k-1}$ (see Lemma \ref{lem2.5}), there is
\begin{eqnarray*}
f(\widetilde{\mu})\leq f_i(\lambda)(\widetilde{\mu}_i-\lambda_i)+f(\lambda),
\end{eqnarray*}
which implies together with \eqref{C2-le1-2}
\begin{eqnarray*}
\sum_{i}f_i(\lambda) (\lambda_i-\mu_i)\leq f(\lambda)-f(\widetilde{\mu})-\epsilon \sum_{i}f_i(\lambda)+2 f_1(\lambda)\lambda_1.
\end{eqnarray*}
Thus, when $\lambda_1$ is large enough, we have by \eqref{C2-le1-11}
\begin{eqnarray*}
\sum_{i}f_i(\lambda) (\lambda_i-\mu_i)\leq -\widetilde{\epsilon}-\epsilon \sum_{i}f_i(\lambda)+2 f_1(\lambda)\lambda_1.
\end{eqnarray*}
If $f_1(\lambda)\lambda_1\leq \frac{\widetilde{\epsilon}}{4}$, we obtain
\begin{eqnarray*}
\sum_{i}f_i(\lambda) (\lambda_i-\mu_i)\leq -\frac{\widetilde{\epsilon}}{2}-\epsilon \sum_{i}f_i(\lambda).
\end{eqnarray*}
Then, we complete the proof if we choose $\theta=\min\{\frac{\widetilde{\epsilon}}{4}, \epsilon\}$.
\end{proof}

\begin{lemma}\label{lem2.6}
Assume $\lambda \in \Gamma_{k-1}$ satisfies \eqref{f} and $\beta_l(z)\geq 0$
($0 \leq l \leq k-2$), then we have
\begin{eqnarray}\label{2.9}
0<\frac{\sigma_l(\lambda)}{\sigma_{k-1}(\lambda)} \leq C, \quad 0 \leq l \leq k-2,
\end{eqnarray}
where the constant $C$ depends on $n$, $k$, $\inf\beta_l$, $\sup |\beta_{k-1}|$. Moreover, we have
\begin{eqnarray}\label{2.10}
-\sup|\beta_{k-1}|\leq\frac{\sigma_k(\lambda)}{\sigma_{k-1}(\lambda)} \leq
C,
\end{eqnarray}
where the constant $C$ depends on $n$, $k$, $\inf\beta_l$ for $0\leq l\leq k-2$,
and $\sum_{i=0}^{k-1}|\beta_i|$.
\end{lemma}

\begin{proof}
On the one hand, if $\frac{\sigma_k}{\sigma_{k-1}}\leq 1$, then we get from the equation \eqref{f}
\begin{equation*}
\beta_l \frac{\sigma_l}{\sigma_{k-1}}
\leq \frac{\sigma_k}{\sigma_{k-1}}-\beta_{k-1} \leq 1+\sup|\beta_{k-1}|, ~~0 \leq l \leq k-2.
\end{equation*}
On the other hand, if $\frac{\sigma_k}{\sigma_{k-1}} > 1$, i.e. $\frac{\sigma_{k-1}}{\sigma_{k}} < 1$.
We can get for $0 \leq l \leq k-2$ by the Newton-MacLaurin inequality \eqref{NM},
\begin{equation*}
\frac{\sigma_l}{\sigma_{k-1}}\leq \frac{(C_n^k)^{k-1-l}C_n^l}{(C_n^{k-1})^{k-l}}
(\frac{\sigma_{k-1}}{\sigma_k})^{k-1-l} \leq \frac{(C_n^k)^{k-1-l}C_n^l}{(C_n^{k-1})^{k-l}} \leq C(n,k).
\end{equation*}
So, we get \eqref{2.9}. Then, it follows that
\begin{equation*}
-\sup|\beta_{k-1}|\leq \frac{\sigma_k}{\sigma_{k-1}}=\sum\limits_{l=0}^{k-2} \beta_l \frac{\sigma_l}
{\sigma_{k-1}}+\beta_{k-1}\leq C \sum_{i=0}^{n}|\beta_i|.
\end{equation*}
Thus, the proof is complete.
\end{proof}

\begin{lemma}\label{lem2.7}
Assume $\lambda \in \Gamma_{k-1}$ and $\beta_l(z)\geq0$
($0 \leq l \leq k-2$), then
\begin{align}
\label{2.11} \sum_{i=1}^{n} f_i\geq \frac{n-k+1}{k} .
\end{align}
\end{lemma}

\begin{proof}
By direct computations, we can get by Propositions \ref{Ell}, \ref{prop2.1} and \ref{prop2.4}
\begin{align*}
\sum_{i=1}^{n} f_i \ge \sum_{i=1}^{n}  {\frac{{\partial \left({\frac{{\sigma _k }}{{\sigma_{k-1}}}} \right)}}{{\partial \lambda_i}}}=&\sum_{i=1}^{n} {\frac{{\sigma _{k-1} (\lambda |i)\sigma_{k-1}-\sigma_k \sigma_{k-2}(\lambda |i)}}{{\sigma_{k-1}^2}}}  \notag \\
=&\frac{{(n-k + 1)\sigma_{k-1} ^2-(n-k+2)\sigma _k \sigma _{k-2}}}{{\sigma _{k-1}^2 }} \notag \\
\ge& \frac{{n-k+1}}{k},
\end{align*}
where we get the last inequality from Newton' inequality \eqref{N-1}. Hence \eqref{2.11} holds.
\end{proof}

\section{$C^0$ estimate}

In this section, we follow the idea of Sun \cite{Sun17}
to derive $C^0$ estimate directly from the cone condition. Moreover, we expect that our estimate
holds true for $\alpha_{k-1}$ with indefinite sign and does not depend on the lower bound of $\alpha_i$ ($0\leq i\leq k-2)$.

\subsection{Some lemmas}

Since $\chi_0 \in \Gamma_{k-1}(M)$, we may assume that there is a uniform constant $\tau>0$ such that
\begin{eqnarray}\label{gama}
\chi_0-\tau\omega \in \Gamma_{k-1}(M) \quad \mbox{and} \quad \omega-\tau\chi_0 \in \Gamma_{k-1}(M).
\end{eqnarray}

Then, we can succeed in extending Lemma 2.3 in \cite{Sun17} to our case.
\begin{lemma}
Let $(M, \omega)$ be a K\"ahler manifold of complex dimension $n\geq2$ and
$\alpha_0(z), ..., \alpha_{k-1}(z)$ be continuous functions on $M$ which satisfy
\begin{eqnarray*}
\alpha_i(z)\geq 0, \ 0\leq i\leq k-2 \quad \mbox{and} \quad \sum_{l=0}^{k-2}\alpha_i(z)>0 \quad \mbox{for all} \quad z \in M.
\end{eqnarray*}
Suppose that
$\chi_0 \in \Gamma_{k-1}(M)$ satisfies
\eqref{gama}. If $u \in C^{2}(M)$ satisfies $\chi_u \in \Gamma_{k-1}(M)$, then we have the following pointwise inequalities:

(1) for $0\leq t\leq 1$ and $1\leq l\leq k-1$
\begin{eqnarray}\label{C0-le-1}
\chi_{tu}^{l-1}\wedge \omega^{n-l}\geq (1-t)^{l-1}
\tau^{l-1}\omega^{n-1}.
\end{eqnarray}

(2) for $0<t\leq 1$ and $1\leq l\leq k-1$,
\begin{eqnarray}\label{C0-le-2}
\chi_{u}^{l}\wedge \omega^{n-l}\leq \frac{1}{t^l}
\chi^{l}_{tu}\wedge \omega^{n-l}.
\end{eqnarray}
Moreover, if $u$ is a solution to the equation \eqref{K-eq} and $\chi_0$ satisfies the cone condition \eqref{cone},
then there exists a uniform constant $C>0$ such that for $0\leq t\leq 1$
\begin{eqnarray}\label{C0-le-3}
&&k\chi_{tu}^{k-1}\wedge \omega^{n-k}-\sum_{l=1}^{k-1}l\alpha_l\chi_{tu}^{l-1}\wedge \omega^{n-l}\\ \nonumber&>&C(1-t)
\chi_{tu}^{k-2}\wedge \omega^{n-k+1},
\end{eqnarray}
and consequently
\begin{eqnarray}\label{C0-le-4}
&&k\chi_{tu}^{k-1}\wedge \omega^{n-k}-\sum_{l=1}^{k-1}l\alpha_l\chi_{tu}^{l-1}\wedge \omega^{n-l}\\ \nonumber&>&C\tau^{k-2}(1-t)^{k-1}
\omega^{n-1}.
\end{eqnarray}
\end{lemma}

\begin{proof}
For $1\leq l\leq k-1$ and $1\leq i\leq n$,
$\sigma_{l}^{\frac{1}{l}}(\lambda)$ and $\sigma_{l-1}^{\frac{1}{l-1}}(\lambda|i)$
are concave in $\Gamma_l$ and $\Gamma_{l-1}$ respectively.  Thus,
\begin{eqnarray}\label{C000-1}
\sigma_{l}^{\frac{1}{l}}(\chi_{tu})
\geq (1-t)\sigma_{l}^{\frac{1}{l}}(\chi_0)+t\sigma_{l}^{\frac{1}{l}}(\chi_u)\geq t\sigma_{l}^{\frac{1}{l}}(\chi_u),
\end{eqnarray}
\begin{eqnarray}\label{C000-2}
\sigma_{l-1}^{\frac{1}{l-1}}(\chi_{tu}|i)
&\geq& (1-t)\sigma_{l-1}^{\frac{1}{l-1}}(\chi_0|i)+t\sigma_{l-1}^{\frac{1}{l-1}}(\chi_u|i)
\nonumber\\&\geq& (1-t)\sigma_{l-1}^{\frac{1}{l-1}}(\chi_0|i),
\end{eqnarray}
and
\begin{eqnarray}\label{C000-3}
\sigma_{l-1}^{\frac{1}{l-1}}(\chi_0|i)
\geq \sigma_{l-1}^{\frac{1}{l-1}}(\chi_0-\tau \omega|i)+\tau\sigma_{l-1}^{\frac{1}{l-1}}(\omega|i)
\geq \tau\sigma_{l-1}^{\frac{1}{l-1}}(\omega|i).
\end{eqnarray}
Combining \eqref{C000-2} and \eqref{C000-3} gives
\begin{eqnarray*}
\sigma_{l-1}(\chi_{tu}|i)\geq
(1-t)^{l-1}\tau^{l-1}\sigma_{l-1}(\chi_0|i),
\end{eqnarray*}
which is just the local form of \eqref{C0-le-1}. Similarly,
\eqref{C0-le-2} is a consequence of \eqref{C000-1}.
Now we only need to prove \eqref{C0-le-3} and \eqref{C0-le-4}.
Since
\begin{eqnarray*}
f(\chi|i):=\frac{\sigma_{k-1}(\chi|i)}{\sigma_{k-2}(\chi|i)}-\sum_{l=1}^{k-2}
\beta_l\frac{\sigma_{l-1}(\chi|i)}{\sigma_{k-2}(\chi|i)}
\end{eqnarray*}
is concave in $\Gamma_{k-2}$ (see Lemma \ref{lem2.5}), we can obtain
\begin{eqnarray}\label{0-1}
f(\chi_{tu}|i)\geq (1-t)f(\chi_{0}|i)
+t f(\chi_{u}|i).
\end{eqnarray}
Moreover, we have by Lemma \ref{k-i}
\begin{eqnarray}\label{0-2}
f(\chi_{u}|i)&=&\frac{\sigma_{k-1}(\chi_{u}|i)}{\sigma_{k-2}(\chi_{u}|i)}
-\sum_{l=1}^{k-2}\beta_l\frac{\sigma_{l-1}(\chi_{u}|i)}{\sigma_{k-2}(\chi_{u}|i)}\nonumber\\&>&\frac{\sigma_{k}(\chi_{u})}{\sigma_{k-1}(\chi_{u})}
-\sum_{l=1}^{k-2}\beta_l\frac{\sigma_{l}(\chi_{u})}{\sigma_{k-1}(\chi_{u})}
\nonumber\\&=&\frac{\beta_0}{\sigma_{k-1}(\chi_{u})}+\beta_{k-1}\nonumber\\&\geq&\beta_{k-1}.
\end{eqnarray}
In addition, since $\chi_0$ satisfies the cone condition \eqref{cone},
there exists some uniform constant $\delta>0$ which is independent of $z$ such that
\begin{eqnarray}\label{0-3}
f(\chi_{0}|i)>\beta_{k-1}(z)+\delta,
\end{eqnarray}
where we write the cone condition in a local version as  Lemma \ref{cone-local}.
Substituting \eqref{0-2} and \eqref{0-3} into \eqref{0-1}, we get
\begin{eqnarray*}
\frac{\sigma_{k-1}(\chi_{tu}|i)}{\sigma_{k-2}(\chi_{tu}|i)}
-\sum_{l=1}^{k-2}\beta_l\frac{\sigma_{l-1}(\chi_{tu}|i)}{\sigma_{k-2}(\chi_{tu}|i)}
>(1-t)(\beta_{k-1}+\delta)+t\beta_{k-1},
\end{eqnarray*}
this is to say
\begin{eqnarray*}
\sigma_{k-1}(\chi_{tu}|i)
-\sum_{l=1}^{k-1}\beta_l\sigma_{l-1}(\chi_{tu}|i)
>\delta(1-t)\sigma_{k-2}(\chi_{tu}|i),
\end{eqnarray*}
which is just the local form of \eqref{C0-le-3}.
Then, \eqref{C0-le-3} follows by substituting \eqref{C0-le-1} into it.
\end{proof}

Next, we derive the following important inequality by iterating the inequality (3.15) in \cite{Sun17}.

\begin{lemma}
Let $(M, \omega)$ be a K\"ahler manifold of complex dimension $n\geq2$. Suppose that
$\chi_0 \in \Gamma_{k-1}(M)$ satisfies
\eqref{gama}. If $u \in C^{2}(M)$ satisfies $\chi_u \in \Gamma_{k-1}(M)$, then we have the following inequalities for $l<k$:
\begin{eqnarray}\label{C0-le-5}
&&\frac{k-1}{l-1}\int_{0}^{\frac{1}{2}}dt\int_{M}e^{-pu}\sqrt{-1}\partial u\wedge \overline{\partial} u\wedge\chi_{tu}^{k-2}\wedge \omega^{n-k+1}\nonumber\\&\geq&\tau^{k-l}
\int_{0}^{\frac{1}{2}}dt\int_{M}e^{-pu}\sqrt{-1}\partial u\wedge \overline{\partial} u\wedge\chi_{tu}^{l-2}\wedge \omega^{n-l+1}.
\end{eqnarray}
\end{lemma}

\begin{proof}
Using integration by parts and Garding's inequality \eqref{G-i}, it yields
\begin{eqnarray*}
&&\int_{0}^{\frac{1}{2}}dt\int_{M}e^{-pu}\sqrt{-1}\partial u\wedge \overline{\partial} u\wedge\chi_{tu}^{l-1}\wedge \omega^{n-l}\\&\geq&
\tau\int_{0}^{\frac{1}{2}}dt\int_{M}e^{-pu}\sqrt{-1}\partial u\wedge \overline{\partial} u\wedge\chi_{tu}^{l-2}\wedge \omega^{n-l+1}
\\&&+\frac{1}{l-1}\int_{0}^{\frac{1}{2}}dt\int_{M}e^{-pu}\sqrt{-1}\partial u\wedge \overline{\partial} u\wedge t\frac{d}{dt}(\chi_{tu}^{l-1}\wedge \omega^{n-l})\\&\geq&
\tau\int_{0}^{\frac{1}{2}}dt\int_{M}e^{-pu}\sqrt{-1}\partial u\wedge \overline{\partial} u\wedge\chi_{tu}^{l-2}\wedge \omega^{n-l+1}
\\&&-\frac{1}{l-1}\int_{0}^{\frac{1}{2}}dt\int_{M}e^{-pu}\sqrt{-1}\partial u\wedge \overline{\partial} u\wedge \chi_{tu}^{l-1}\wedge \omega^{n-l}.
\end{eqnarray*}
Thus,
\begin{eqnarray*}
&&\frac{l}{l-1}\int_{0}^{\frac{1}{2}}dt\int_{M}e^{-pu}\sqrt{-1}\partial u\wedge \overline{\partial} u\wedge\chi_{tu}^{l-1}\wedge \omega^{n-l}\\&\geq&
\tau\int_{0}^{\frac{1}{2}}dt\int_{M}e^{-pu}\sqrt{-1}\partial u\wedge \overline{\partial} u\wedge\chi_{tu}^{l-2}\wedge \omega^{n-l+1}
.
\end{eqnarray*}
So, we obtain by iteration
\begin{eqnarray*}
&&\frac{k-1}{l-1}\int_{0}^{\frac{1}{2}}dt\int_{M}e^{-pu}\sqrt{-1}\partial u\wedge \overline{\partial} u\wedge\chi_{tu}^{k-2}\wedge \omega^{n-k+1}\\&\geq&\tau^{k-l}
\int_{0}^{\frac{1}{2}}dt\int_{M}e^{-pu}\sqrt{-1}\partial u\wedge \overline{\partial} u\wedge\chi_{tu}^{l-2}\wedge \omega^{n-l+1}
,
\end{eqnarray*}
which complete the proof.
\end{proof}

Once we have established the following inequality \eqref{C0}, we can use it to derive $C^0$
estimate \eqref{C00-1} by the standard argument in \cite{Tos19, Tos10} without using the equation \eqref{K-eq}.

\begin{lemma}\label{C0-00}
Let $(M, \omega)$ be a closed K\"ahler manifold of complex dimension $n\geq2$ and
$\alpha_0(z), ..., \alpha_{k-1}(z)$ be $C^2$ functions on $M$ which satisfy
\begin{eqnarray*}
\alpha_i(z)\geq 0, \ 0\leq i\leq k-2 \quad \mbox{and} \quad \sum_{l=0}^{k-2}\alpha_i(z)>0 \quad \mbox{for all} \quad z \in M.
\end{eqnarray*}
Suppose that $\chi_0 \in \Gamma_{k-1}(M)$ is closed and satisfies
\eqref{gama}. If $u \in C^{2}(M)$ satisfies $\chi_u \in \Gamma_{k-1}(M)$, then there exists uniform constants $C$
and $p_0$ such that for $p\geq p_0$ the following inequality holds
\begin{eqnarray}\label{C0}
\int_{M}|\partial e^{-\frac{p}{2}u}|^{2}_{g}\omega^n\leq C p\int_{M}e^{-pu}\omega^n.
\end{eqnarray}
Thus, there exists a uniform constant $C$ depends on the given
data $M$, $\omega$, $\chi_0$, and $|\alpha_i|_{C^0(M)}$ ($0\leq i\leq k-1$) such that
\begin{eqnarray}\label{C00-1}
|u|_{C^0(M)}\leq C \quad \mbox{with} \quad \sup_{M}u=0.
\end{eqnarray}
\end{lemma}

\begin{proof}
Since
\begin{eqnarray*}
&&(\chi_{u}^k\wedge \omega^{n-k}-\chi_{0}^k\wedge \omega^{n-k})
-\sum_{l=0}^{k-1}\alpha_{l}(\chi_{u}^l\wedge \omega^{n-l}-\chi_{0}^l\wedge \omega^{n-l})
\\&=&\int_{0}^{1}\sqrt{-1}\partial\overline{\partial} u\wedge\Big(k\chi_{tu}^{k-1}\wedge \omega^{n-k}
-\sum_{l=1}^{k-1}l\alpha_{l}\chi_{tu}^{l-1}\wedge \omega^{n-l}\Big)dt,
\end{eqnarray*}
we have
\begin{eqnarray}\label{C0-1}
&&\int_{M}e^{-pu}\Big[(\chi_{u}^{k}\wedge \omega^{n-k}-\chi_{0}^k\wedge \omega^{n-k})
-\sum_{l=0}^{k-1}\alpha_{l}(\chi_{u}^{l}\wedge \omega^{n-l}-\chi_{0}^l\wedge \omega^{n-l})\Big]
\nonumber\\ \nonumber
&=&p\int_{0}^{1}dt\int_{M}e^{-pu}\sqrt{-1}\partial u\wedge \overline{\partial} u\wedge\Big(k\chi_{tu}^{k-1}\wedge \omega^{n-k}
-\sum_{l=1}^{k-1}l\alpha_{l}\chi_{tu}^{l-1}\wedge \omega^{n-l}\Big)\\ \nonumber&&
-\sum_{l=1}^{k-1}\frac{l}{p}\int_{0}^{1}dt\int_{M}e^{-pu}\sqrt{-1}\partial \overline{\partial}\alpha_l \wedge\chi_{tu}^{l-1}\wedge \omega^{n-l}
\nonumber \\ \nonumber&\geq&p\int_{0}^{1}dt\int_{M}e^{-pu}\sqrt{-1}\partial u\wedge \overline{\partial} u\wedge\Big(k\chi_{tu}^{k-1}\wedge \omega^{n-k}
-\sum_{l=1}^{k-1}l\alpha_{l}\chi_{tu}^{l-1}\wedge \omega^{n-l}\Big)\\ &&
-\sum_{l=1}^{k-1}\frac{C_l}{p}\int_{0}^{1}dt\int_{M}e^{-pu}\chi_{tu}^{l-1}\wedge \omega^{n-l+1}
.
\end{eqnarray}
Notice that
\begin{eqnarray}\label{C0-2}
&&\int_{M}e^{-pu}\Big[(\chi_{u}^{k}\wedge \omega^{n-k}-\chi_{0}^k\wedge \omega^{n-k})
-\sum_{l=0}^{k-1}\alpha_{l}(\chi_{u}^{l}\wedge \omega^{n-l}-\chi_{0}^l\wedge \omega^{n-l})\Big]
\\&=& \int_{M}e^{-pu}\Big(-\chi_{0}^k\wedge \omega^{n-k}
+\sum_{l=0}^{k-1}\alpha_{l}\chi_{0}^l\wedge \omega^{n-l}\Big)\nonumber \\ \nonumber&\leq&C\int_{M}e^{-pu}\omega^{n}.
\end{eqnarray}
Using the inequality \eqref{C0-le-2}, we obtain
\begin{eqnarray}\label{C0-3}
&&\int_{0}^{1}dt\int_{M}e^{-pu}\chi_{tu}^{l-1}\wedge \omega^{n-l+1}
\nonumber\\&\leq&2^{l-1}\int_{0}^{1}dt\int_{M}e^{-pu}\chi_{\frac{tu}{2}}^{l-1}\wedge \omega^{n-l+1}\nonumber \\&\leq&2^{l}\int_{0}^{\frac{1}{2}}dt\int_{M}e^{-pu}\chi_{tu}^{l-1}\wedge \omega^{n-l+1}.
\end{eqnarray}
Plugging the inequalities \eqref{C0-2} and \eqref{C0-3} into \eqref{C0-1}, it yields
\begin{eqnarray}\label{C0-4}
&&p\int_{0}^{1}dt\int_{M}e^{-pu}\sqrt{-1}\partial u\wedge \overline{\partial} u\wedge\Big(k\chi_{tu}^{k-1}\wedge \omega^{n-k}
-\sum_{l=1}^{k-1}l\alpha_{l}\chi_{tu}^{l-1}\wedge \omega^{n-l}\Big)\\ \nonumber&\leq&
\sum_{l=1}^{k-1}\frac{2^{l}C_l}{p}\int_{0}^{\frac{1}{2}}dt\int_{M}e^{-pu}\chi_{tu}^{l-1}\wedge \omega^{n-l+1}
+C\int_{M}e^{-pu}\omega^{n}.
\end{eqnarray}
We deal with the first term on the right side of the inequality \eqref{C0-4} according to the inequality \eqref{C0-le-5}
\begin{eqnarray}\label{C0-5}
&&\frac{1}{p}\int_{0}^{\frac{1}{2}}dt\int_{M}e^{-pu}\chi_{tu}^{l-1}\wedge \omega^{n-l+1}
\nonumber \\ \nonumber&=&(l-1)\int_{0}^{\frac{1}{2}}dt\int_{0}^{t}ds\int_{M}e^{-pu}
\sqrt{-1}\partial u\wedge \overline{\partial} u\wedge\chi_{su}^{l-2}\wedge \omega^{n-l+1}\nonumber \\&&+
\frac{1}{2p}\int_{M}e^{-pu}\chi_{0}^{l-1}\wedge \omega^{n-l+1}\nonumber \\&\leq&\frac{l-1}{2}\int_{0}^{\frac{1}{2}}dt\int_{M}e^{-pu}
\sqrt{-1}\partial u\wedge \overline{\partial} u\wedge\chi_{tu}^{l-2}\wedge \omega^{n-l+1}\nonumber \\&&+
\frac{1}{2p}\int_{M}e^{-pu}\chi_{0}^{l-1}\wedge \omega^{n-l+1}\nonumber \\
&\leq&\frac{k-1}{\tau^{k-l}}\int_{0}^{\frac{1}{2}}dt\int_{M}e^{-pu}
\sqrt{-1}\partial u\wedge \overline{\partial} u\wedge\chi_{tu}^{k-2}\wedge \omega^{n-k+1}\\ \nonumber &&+
\frac{C}{2p}\int_{M}e^{-pu}\omega^{n}.
\end{eqnarray}
To cancel the first term in the right hand of \eqref{C0-5}, we will use part of the left term in \eqref{C0-4}. In
details, we can get the following positive term for $0\leq t\leq \frac{1}{2}$ from the inequality \eqref{C0-le-3}
\begin{eqnarray}\label{C0-6}
&&pe^{-pu}\sqrt{-1}\partial u\wedge \overline{\partial} u\wedge\Big(k\chi_{tu}^{k-1}\wedge \omega^{n-k}
-\sum_{l=1}^{k-1}l\alpha_{l}\chi_{tu}^{l-1}\wedge \omega^{n-l}\Big)\nonumber\\&\geq&
Cpe^{-pu}\sqrt{-1}\partial u\wedge \overline{\partial} u\wedge\chi_{tu}^{k-2}\wedge \omega^{n-k+1}.
\end{eqnarray}
Thus, if we choose $p$ sufficiently large, the integral of the term \eqref{C0-6} on $M$ can
kill the first term in the right hand of \eqref{C0-5}. Then, \eqref{C0-4} becomes
\begin{eqnarray*}\label{C0-7}
&&\frac{p}{2}\int_{0}^{1}dt\int_{M}e^{-pu}\sqrt{-1}\partial u\wedge \overline{\partial} u\wedge\Big(k\chi_{tu}^{k-1}\wedge \omega^{n-k}
-\sum_{l=0}^{k-1}l\alpha_{l}\chi_{tu}^{l-1}\wedge \omega^{n-l}\Big)\nonumber\\&\leq&
C\int_{M}e^{-pu}\omega^{n},
\end{eqnarray*}
which implies in view of \eqref{C0-le-4}
\begin{eqnarray*}\label{C0-7}
p \int_{M}e^{-pu}\sqrt{-1}\partial u\wedge \overline{\partial} u\wedge \omega^{n-1}\leq C\int_{M}
e^{-pu}\omega^{n}.
\end{eqnarray*}
So, our proof is completed.
\end{proof}

\section{$C^2$ estimate}

\subsection{Notations and some lemmas}

In local complex coordinates $(z^1, ..., z^n)$, the subscripts of a function $u$ always denote the covariant derivatives of $u$
with respect to $\omega$ in the directions of the local frame $\frac{\partial }{\partial z^1}, ..., \frac{\partial}{\partial z^n}$.
Namely,
\begin{eqnarray*}
u_i=\nabla_{\frac{\partial}{\partial z^i}}u,
\quad u_{i\overline{j}}=\nabla_{\frac{\partial}{\partial \overline{z}^j}}
\nabla_{\frac{\partial}{\partial z^i}}u, \quad u_{i\overline{j}k}=\nabla_{\frac{\partial}{\partial z^k}}\nabla_{\frac{\partial}{\partial \overline{z}^j}}
\nabla_{\frac{\partial}{\partial z^i}}u.
\end{eqnarray*}
But, the covariant derivatives of a $(1, 1)$-form $\chi$
with respect to $\omega$ will be denoted by indices with
semicolons, e.g.,
\begin{eqnarray*}
\chi_{i\overline{j}; k}=\nabla_{\frac{\partial}{\partial z^k}}
\chi(\frac{\partial}{\partial z^i}, \frac{\partial}{\partial \overline{z}^j}), \quad
\chi_{i\overline{j}; k\overline{l}}=\nabla_{\frac{\partial}{\partial \overline{z}^{l}}}\nabla_{\frac{\partial}{\partial z^k}}
\chi(\frac{\partial}{\partial z^i}, \frac{\partial}{\partial \overline{z}^j}).
\end{eqnarray*}

We recall the following commutation formula on K\"ahler manifolds $(M, \omega)$ \cite{Hou10, Guan10, Guan12}.

\begin{lemma}\label{2rd}
For $u \in C^4(M)$, we have
\begin{eqnarray*}
u_{i\overline{j}l}=u_{i\overline{l}j}-u_p R_{l \overline{j} i}^{\ \ p}, \quad
u_{p\overline{j}\overline{m}}=u_{p\overline{m}\overline{j}}, \quad
u_{i\overline{q}l}=u_{l\overline{q}i},
\end{eqnarray*}
\begin{eqnarray*}
u_{i\overline{j}l\overline{m}}=
u_{l\overline{m}i\overline{j}}+u_{p\overline{j}}R_{l\overline{m}i}^{\ \ \ p}
-u_{p\overline{m}}R_{i\overline{j}l}^{\ \ \ p},
\end{eqnarray*}
where $R$ is the curvature tensor of $(M, \omega)$.
\end{lemma}

For the convenience of notations, we will denote
\begin{equation*}
F_k(\chi):= \frac{\sigma_k(\chi)}{\sigma_{k-1}(\chi)},
\ \ F_l(\chi) := -\frac{\sigma_l(\chi)}{\sigma_{k-1}(\chi)},~ 0\leq l\leq k-2,
\end{equation*}
\begin{equation*}
F(\chi, z):= F_k(\chi) + \sum_{l=0}^{k-2} \beta_l(z) F_l(\chi),
\end{equation*}
and
\begin{equation*}
F^{i\overline{j}} :=\frac{\partial F}{\partial \chi_{i\overline{j}}}, \quad
F^{i\overline{j}, k\overline{l}} :=\frac{\partial^2 F}{\partial \chi_{i\overline{j}}
\partial \chi_{r\overline{s}}},\quad
F_{l}^{i\overline{j}} :=\frac{\partial F_l}{\partial \chi_{i\overline{j}}}, \quad
F_{l}^{i\overline{j}, r\overline{s}} :=\frac{\partial^2 F}{\partial \chi_{i\overline{j}}
\partial \chi_{r\overline{s}}},
\end{equation*}
where $1 \leq i, j, r, s\leq n$ and $0\leq l\leq k-2.$

\begin{lemma}\label{C^2-1}
Let $u \in C^4(M)$ be a solution to the equation \eqref{K-eq2} on $(M, \omega)$ and $\chi \in \Gamma_{k-1}(M)$.
For any $z \in M$, we choose a normal coordinate such
that at this point $\omega=\frac{\sqrt{-1}}{2}\delta_{ij}dz^i \wedge d\overline{z}^j$.
Then, we have for any $\delta<1$ at $z$
\begin{eqnarray}\label{C^2-1}
F^{i\overline{j}}\chi_{i\overline{j}; p\overline{p}}&\geq&-(1-\delta^2)F^{i\overline{j},r\overline{s}}\chi_{i\overline{j};
p}\chi_{r\overline{s}; \overline{p}}+\nabla_{\overline{p}}\nabla_p\beta_{k-1}
\nonumber\\&&+\sum_{l=0}^{k-2}\frac{1}{1+\frac{1}{k+1-l}}\frac{|\nabla_p\beta_l|^2}{\delta^2\beta_l}F_l-\sum_{l=0}^{k-2}
\nabla_{\overline{p}}\nabla_p\beta_{l}\cdot F_{l}.
\end{eqnarray}
\end{lemma}

\begin{proof}
Differentiating the equation \eqref{K-eq2} once, we have
\begin{eqnarray*}
\nabla_p\beta_{k-1}=F^{i\overline{j}}\chi_{i\overline{j}; p}
+\sum_{l=0}^{k-2}\nabla_p \beta_{l}\cdot F_{l}.
\end{eqnarray*}
Differentiating the equation \eqref{K-eq2} again, we obtain
\begin{eqnarray*}
\nabla_{\overline{p}}\nabla_p\beta_{k-1}&=&F^{i\overline{j},r\overline{s}}\chi_{i\overline{j};
p}\chi_{r\overline{s}; \overline{p}} +F^{i\overline{j}}\chi_{i\overline{j};
p\overline{p}}+\sum_{l=0}^{k-2}\Big(\nabla_{\overline{p}}\beta_{l}\cdot F_{l}^{i\overline{j}}\chi_{i\overline{j};
p}+\nabla_{p}\beta_{l} \cdot F_{l}^{i\overline{j}}\chi_{i\overline{j};
\overline{p}}\Big)\\&&+\sum_{l=0}^{k-2}\nabla_{\overline{p}}\nabla_p\beta_{l} \cdot F_{l}.
\end{eqnarray*}
Moreover, since the operator $(\frac{\sigma_{k-1}}{\sigma_l})^{\frac{1}{k-1-l}}$
is concave for $0\leq l\leq k-2$ (see Proposition \ref{Con}), we have
\begin{eqnarray}
-F_l^{i\overline{j},r\overline{s}}\chi_{i\overline{j};p}\chi_{r\overline{s};\overline{p}} \ge
-\big(1+\frac{1}{k-1-l}\big)F_l^{-1}F_l^{i\overline{j}}F_l^{r\overline{s}}\chi_{i\overline{j};p}\chi_{r\overline{s};
\overline{p}}.\label{93}
\end{eqnarray}
Since $F_k$ is concave in $\Gamma_{k-1}$, throwing away the negative term
$$\delta^2\frac{\partial^2 F_{k}}{\partial \chi_{i\overline{j}}
\partial \chi_{r\overline{s}}}
\chi_{i\overline{j};p}\chi_{r\overline{s}
;\overline{p}}$$
gives
\begin{eqnarray*}
&&\nabla_{\overline{p}}\nabla_p\beta_{k-1}-(1-\delta^2)F^{i\overline{j},r\overline{s}}
\chi_{i\overline{j};p}\chi_{r\overline{s}
;\overline{p}}\\&\leq&\sum_{l=0}^{k-2}\delta^2\beta_lF_{l}^{i\overline{j},r\overline{s}}
\chi_{i\overline{j};p}\chi_{r\overline{s}
;\overline{p}}+F^{i\overline{j}}\chi_{i\overline{j};
p\overline{p}}+\sum_{l=0}^{k-2}\Big(\nabla_{\overline{p}}\beta_{l}\cdot F_{l}^{i\overline{j}}\chi_{i\overline{j};
p}+\nabla_{p}\beta_{l}\cdot F_{l}^{i\overline{j}}\chi_{i\overline{j};
\overline{p}}\Big)\\&&+\sum_{l=0}^{k-2}\nabla_{\overline{p}}\nabla_p\beta_{l}\cdot F_{l}\\&\leq&
\delta^2\sum_{l=0}^{k-2}\beta_l\big(1+\frac{1}{k-1-l}\big)F_l^{-1}|F_l^{i\overline{j}}\chi_{i\overline{j}
;p}|^2+F^{i\overline{j}}\chi_{i\overline{j};
p\overline{p}}\\&&+\sum_{l=0}^{k-2}\Big(\nabla_{\overline{p}}\beta_{l} \cdot F_{l}^{i\overline{j}}\chi_{i\overline{j};
p}+\nabla_{p}\beta_{l}\cdot F_{l}^{i\overline{j}}\chi_{i\overline{j};
\overline{p}}\Big)+\sum_{l=0}^{k-2}\nabla_{\overline{p}}\nabla_p\beta_{l} \cdot F_{l}
\\&=&\frac{\delta^2(k-l)}{k-1-l}
\sum_{l=1}^{k-2}\beta_lF_l^{-1}\bigg|F_l^{i\overline{j}}\chi_{i\overline{j};
p}+\frac{1}{1+\frac{1}{k-1-l}}\frac{\nabla_p\beta_l}{\delta^2 \beta_l}F_l\bigg|^2\\&&-
\sum_{l=0}^{k-2}\frac{1}{1+\frac{1}{k-1-l}}\frac{|\nabla_p\beta_l|^2}{\delta^2\beta_l}F_l+F^{i\overline{j}}\chi_{i\overline{j};
p\overline{p}}+\sum_{l=0}^{k-2}\nabla_{\overline{p}}\nabla_p\beta_{l}\cdot F_{l}\\&\leq&
-\sum_{l=0}^{k-2}\frac{1}{1+\frac{1}{k-1-l}}\frac{|\nabla_p\beta_l|^2}{\delta^2\beta_l}F_l+F^{i\overline{j}}\chi_{i\overline{j};
p\overline{p}}+\sum_{l=0}^{k-2}\nabla_{\overline{p}}\nabla_p\beta_{l} \cdot F_{l},
\end{eqnarray*}
where we used the inequality \eqref{93} to get the second inequality.
So, our proof is completed.
\end{proof}

\subsection{$C^2$ estimate}

\begin{lemma}
Let $(M, \omega)$ be a closed K\"ahler manifold of complex dimension $n\geq2$.
Suppose that $\chi_0$ satisfies the cone condition \eqref{cone} and $\alpha_0(z)$,
$\alpha_1(z)$, $...$, $\alpha_{k-1}(z)$ are real $C^2$ functions on $M$
satisfying (i) and (ii) in Assumption \ref{Ass}.
If $u \in C^4(M)$ is a solution to the equation \eqref{K-eq2} with $\chi_u \in \Gamma_{k-1}(M)$, we have
\begin{eqnarray*}
\sup_{M}|\partial \overline{\partial} u|\leq C(\sup_{M}|\nabla u|^2+1),
\end{eqnarray*}
where the constant $C$ depends on $M, \omega, \chi_0$, $|u|_{C^0}$, $\inf_{M}\alpha_l$ ($0\leq l\leq k-2$) and $|\alpha_i|_{C^2}$ ($0\leq i\leq k-1$).
\end{lemma}

\begin{proof}
For convenience, in the following argument, we write $\chi=\chi_u$ for short (suppressing the subscript $u$).
Following the work of Hou-Ma-Wu \cite{Hou10}, we define for $z \in M$
and a unit vector $\xi \in T_{z}^{(1, 0)}M$
\begin{eqnarray*}
W(z, \xi)=\ln (\chi_{ij}\xi^i\xi^j)+\varphi(|\nabla u|^2)+\psi(u),
\end{eqnarray*}
where
\begin{eqnarray*}
\varphi(s)=-\frac{1}{2}\log \Big(1-\frac{s}{2K}\Big) \quad \mbox{for} \quad 0\leq s\leq K-1
\end{eqnarray*}
and
\begin{eqnarray*}
\psi(t)=-A\log \Big(1+\frac{t}{2L}\Big) \quad \mbox{for} \quad -L+1\leq t\leq 0.
\end{eqnarray*}
Here, we set
\begin{eqnarray*}
K=:\sup_{M}|\nabla u|^2+1, \quad L:=\sup_{M}|u|+1, \quad  A:=2L\Lambda,
\end{eqnarray*}
and $\Lambda$ is a large constant that we will choose later. Clearly, $\varphi$ satisfies
\begin{eqnarray*}
\frac{1}{2K}\geq \varphi^{\prime}\geq \frac{1}{4K}, \quad  \varphi^{\prime \prime}=2(\varphi^{\prime})^2>0,
\end{eqnarray*}
and $\psi$ satisfies the bounds
\begin{eqnarray*}
2\Lambda\geq -\psi^{\prime}\geq \Lambda, \quad  \psi^{\prime \prime}
\geq\frac{2\varepsilon}{1-\varepsilon}(\psi^{\prime})^2 \quad \mbox{for all} \quad \varepsilon\leq\frac{1}{2A+1}.
\end{eqnarray*}

The function $W$ must achieve its maximum at some point $z$ in some unit direction
of $\eta$. Around $z$, we choose a normal chart such that $\chi$ is diagonal and
\begin{eqnarray}\label{chi}
\chi_{\eta\overline{\eta}}=\chi_{1\overline{1}}\geq ...\geq \chi_{n\overline{n}}.
\end{eqnarray}
It follows from \eqref{f-3}
\begin{eqnarray}\label{F^i}
F^{1\overline{1}}\leq ...\leq F^{n\overline{n}}.
\end{eqnarray}
Furthermore, we arrive at $z$
\begin{eqnarray}\label{Diff-1}
W_i=\frac{\chi_{1\overline{1};i}}{\chi_{1\overline{1}}}
+\varphi^{\prime}\nabla_i(|\nabla u|^2)+\psi^{\prime}u_i=0
\end{eqnarray}
and
\begin{eqnarray}\label{Diff-2}
W_{i\overline{i}}&=&\frac{\chi_{1\overline{1};i\overline{i}}}{\chi_{1\overline{1}}}-\frac{|\chi_{1\overline{1};i}|^2}{\chi_{1\overline{1}}^2}
+\varphi^{\prime\prime}|\nabla_i(|\nabla u|^2)|^2
+\varphi^{\prime}\nabla_{\overline{i}}\nabla_i(|\nabla u|^2)\nonumber\\&&+\psi^{\prime\prime}|u_i|^2+\psi^{\prime}u_{i\overline{i}}\leq 0.
\end{eqnarray}
Commuting derivatives, we have the identity from Lemma \ref{2rd}
\begin{eqnarray*}
\chi_{1\overline{1}; i\overline{i}}=\chi_{i\overline{i}; 1\overline{1}}
+\chi_{1\overline{1}}R_{1\overline{1} i\overline{i}}-\chi_{i\overline{i}}
R_{i\overline{i} 1\overline{1}}.
\end{eqnarray*}
Using the above inequality, we get from \eqref{C^2-1}
\begin{eqnarray*}
F^{i\overline{j}}\chi_{1\overline{1}; i\overline{j}}&\geq&-(1-\delta^2)F^{i\overline{j},r\overline{s}}\chi_{i\overline{j};
1}\chi_{r\overline{s}; \overline{1}}-C\chi_{1\overline{1}}\sum_{i}F^{i\overline{i}}-CF_k-C\sum_{l=0}^{k-1}|F_l|\\
&&+\nabla_{\overline{1}}\nabla_1\beta_{k-1}
-\sum_{l=0}^{k-2}\frac{1}{1+\frac{1}{k+1-l}}\frac{|\nabla_1\beta_l|^2}{\delta^4\beta_l}F_l-\sum_{l=0}^{k-2}
\nabla_{\overline{1}}\nabla_1\beta_{l}\cdot F_{l}\\&\geq&-(1-\delta^2)F^{i\overline{j},r\overline{s}}\chi_{i\overline{j};
1}\chi_{r\overline{s}; \overline{1}}-C\chi_{1\overline{1}}\sum_{i}F^{i\overline{i}}-C,
\end{eqnarray*}
where we used the equality $$F^{i\overline{i}}\chi_{i\overline{i}}=F_k+\sum_{l=0}^{k-2}(l-k+1)\beta_lF_l$$ and
the fact that $F_k$ and $F_l$ ($0\leq l\leq k-2$) are bounded (see \eqref{2.9} and \eqref{2.10}).

Multiplying \eqref{Diff-2} by $F^{i\overline{i}}$ and summing it over index $i$, we can know from the above inequality
\begin{eqnarray}\label{C2-3}
0&\geq&-\frac{F^{i\overline{i}}|\chi_{1\overline{1};i}|^2}{\chi_{1\overline{1}}^2}
-(1-\delta^2)\frac{1}{\chi_{1\overline{1}}}F^{i\overline{j},r\overline{s}}\chi_{i\overline{j};
1}\chi_{r\overline{s}; \overline{1}}
\nonumber \\&&+\varphi^{\prime\prime}F^{i\overline{i}}|\nabla_i(|\nabla u|^2)|^2
+\varphi^{\prime}F^{i\overline{i}}\nabla_{\overline{i}}\nabla_{i}(|\nabla u|^2)
+\psi^{\prime\prime}F^{i\overline{i}}|u_i|^2+\psi^{\prime}F^{i\overline{i}}u_{i\overline{i}}  \nonumber \\&&-C\sum_{i}F^{i\overline{i}}-\frac{C}{\chi_{1\overline{1}}}.
\end{eqnarray}
To proceed, we need the following calculation
\begin{eqnarray*}
\nabla_i(|\nabla u|^2)=\sum_{j}(-u_j\chi_{0i\overline{j}}+u_{ji}u_{\overline{j}})+u_i\chi_{ii}
\end{eqnarray*}
and
\begin{eqnarray*}
\nabla_{\overline{i}}\nabla_i(|\nabla u|^2)
&=&\sum_{j}(\chi_{0j\overline{i}}\chi_{0i\overline{j}}+u_{ji}u_{\overline{j}\overline{i}}
-\chi_{0i\overline{i};j}u_{\overline{j}}-u_j\chi_{0i\overline{i}; \overline{j}})
+\sum_{j, k}R_{i\overline{i}j\overline{k}}u_ku_{\overline{j}}\\&&+\chi_{i\overline{i}}^{2}
-2\chi_{0i\overline{i}}\chi_{i\overline{i}}+2\sum_{j} \mathrm{Re}\{\chi_{i\overline{i}; j}u_{\overline{j}}\}
\\&\geq&-2\sum_{j}\mathrm{Re}\{\chi_{0i\overline{i}; j}u_{\overline{j}}\}
+\sum_{j, k}R_{i\overline{i}j\overline{k}}u_ku_{\overline{j}}+\frac{1}{2}\chi_{i\overline{i}}^{2}
-2\chi_{0i\overline{i}}^{2}\\&&+2\sum_{j} \mathrm{Re}\{\chi_{i\overline{i}; j}u_{\overline{j}}\},
\end{eqnarray*}
where $\mathrm{Re} \{z\}$ denotes the real part of the complex number $z$.
Thus, we can estimate the term
\begin{eqnarray}\label{C2-3-1}
&&\varphi^{\prime}\sum_{i}F^{i\overline{i}}\nabla_{\overline{i}}\nabla_i(|\nabla u|^2)
\nonumber\\&\geq&-C\frac{|\nabla u|+|\nabla u|^2}{K}\sum_{i}F^{i\overline{i}}
+\frac{1}{2}\varphi^{\prime}\sum_{i}F^{i\overline{i}}\chi_{i\overline{i}}^{2}
-2\varphi^{\prime}\sum_{i}F^{i\overline{i}}\chi_{0i\overline{i}}^{2}
\nonumber\\&&-2\varphi^{\prime}\sum_{j} \mathrm{Re}\{\sum_{l=0}^{k-1}(\nabla_j \beta_l \cdot F_l)u_{\overline{j}}\}
\nonumber\\&\geq&-C\frac{1+|\nabla u|+|\nabla u|^2}{K}\sum_{i}F^{i\overline{i}}
+\frac{1}{2}\varphi^{\prime}\sum_{i}F^{i\overline{i}}\chi_{i\overline{i}}^{2}.
\end{eqnarray}
where we used the following equality
\begin{eqnarray}\label{C^2-0}
F^{i\overline{j}}\chi_{i\overline{j}; p}=-\sum_{l=0}^{k-1}\nabla_p \beta_l \cdot F_l
\end{eqnarray}
to get the first inequality and \eqref{C^2-0} can be directly deduced by \eqref{K-eq2}.

Taking the inequality \eqref{C2-3-1} into \eqref{C2-3}, it yields
\begin{eqnarray}\label{C2-4}
0&\geq&-\frac{F^{i\overline{i}}|\chi_{1\overline{1};i}|^2}{\chi_{1\overline{1}}^2}
-(1-\delta^2)\frac{1}{\chi_{1\overline{1}}}F^{i\overline{j},r\overline{s}}\chi_{i\overline{j};
1}\chi_{r\overline{s}; \overline{1}}
\nonumber \\&&+\varphi^{\prime\prime}F^{i\overline{i}}|\nabla_i(|\nabla u|^2)|^2
+\frac{\varphi^{\prime}}{2}F^{i\overline{i}}\chi_{i\overline{i}}^2
+\psi^{\prime\prime}F^{i\overline{i}}|u_i|^2+\psi^{\prime}F^{i\overline{i}}u_{i\overline{i}}
\nonumber \\&&-C\sum_{i}F^{i\overline{i}}-C.
\end{eqnarray}
Now, we divide our proof into two cases separately, depending on whether
$\chi_{n\overline{n}}<-\delta\chi_{1\overline{1}}$ or not, for a small $\delta$ to be chosen later.

\textbf{Case 1.} $\chi_{n\overline{n}}<-\delta\chi_{1\overline{1}}$. In this case, it follows that
$\chi_{1\overline{1}}^{2}\leq \frac{1}{\delta^2}\chi_{n\overline{n}}^{2}$.
So, we only need to bound $\chi_{n\overline{n}}^{2}$. Clearly, we can obtain
if we throw some positive terms in \eqref{C2-4}
\begin{eqnarray}\label{C2-5}
0&\geq&-\frac{F^{i\overline{i}}|\chi_{1\overline{1};i}|^2}{\chi_{1\overline{1}}^2}
+\varphi^{\prime\prime}F^{i\overline{i}}|\nabla_i(|\nabla u|^2)|^2
+\frac{\varphi^{\prime}}{2}F^{i\overline{i}}\chi_{i\overline{i}}^2
+\psi^{\prime}F^{i\overline{i}}u_{i\overline{i}}
\nonumber \\&&-C\sum_{i}F^{i\overline{i}}-C.
\end{eqnarray}
Note that we get from \eqref{uii}
\begin{eqnarray}\label{C2-6}
\psi^{\prime}\sum_{i}F^{i\overline{i}}u_{i\overline{i}}&=&\psi^{\prime}\sum_{i}F^{i\overline{i}}
(\chi_{i\overline{i}}-(\chi_{0i\overline{i}}-\tau)-\tau)\nonumber
\\&=&\psi^{\prime}\Big(F_k+\sum_{l=0}^{k-1}(l-k+1)\beta_lF_l-\sum_{i}F^{i\overline{i}}
(\chi_{0i\overline{i}}-\tau)-\tau\sum_{i}F^{i\overline{i}}\Big)\nonumber\\&\geq&-C\Lambda+\tau\Lambda\sum_{i}F^{i\overline{i}}
\end{eqnarray}
Plugging \eqref{C2-6} into \eqref{C2-5}, we get by choosing $\Lambda$ large enough
\begin{eqnarray}\label{C2-7}
\varphi^{\prime\prime}F^{i\overline{i}}|\nabla_i(|\nabla u|^2)|^2
+\frac{\varphi^{\prime}}{2}F^{i\overline{i}}\chi_{i\overline{i}}^2
&\leq&\frac{F^{i\overline{i}}|\chi_{1\overline{1};i}|^2}{\chi_{1\overline{1}}^2}
+C(1+\Lambda).
\end{eqnarray}
In addition, we have from \eqref{Diff-1}
\begin{eqnarray}\label{C2-8}
\sum_{i}\frac{F^{i\overline{i}}|\chi_{1\overline{1};i}|^2}{\chi_{1\overline{1}}^2}
&=&\sum_{i}F^{i\overline{i}}|\varphi^{\prime}\nabla_i(|\nabla u|^2)+\psi^{\prime}u_i|^2
\nonumber\\&\leq&2(\varphi^{\prime})^2\sum_{i}F^{i\overline{i}}|\nabla_i(|\nabla u|^2)|^2
+8\Lambda^2K\sum_{i}F^{i\overline{i}}.
\end{eqnarray}
Moreover, we obtain in view of \eqref{chi} and \eqref{F^i}
\begin{eqnarray}\label{C2-91}
\sum_{i}F^{i\overline{i}}\chi_{i\overline{i}}^2\geq F^{n\overline{n}}\chi_{n\overline{n}}^2
\geq\frac{1}{n}\chi_{n\overline{n}}^2\sum_{i} F^{i\overline{i}}.
\end{eqnarray}
Substituting \eqref{C2-8} and \eqref{C2-91} into \eqref{C2-7}, we get
\begin{eqnarray*}
\frac{1}{K}\chi_{n\overline{n}}^2\sum_{i} F^{i\overline{i}}\leq 8\Lambda^2K\sum_{i}F^{i\overline{i}}+C(1+\Lambda).
\end{eqnarray*}
Then, it follows that by \eqref{2.11}
\begin{eqnarray*}
\chi_{1\overline{1}}\leq C K.
\end{eqnarray*}

\textbf{Case 2. }$\chi_{n\overline{n}}\geq-\delta \chi_{1\overline{1}}$.
Define the set
\begin{eqnarray*}
I=\Big\{i \in \{1, 2, ..., n\}: F^{i\overline{i}}>\delta^{-1}F^{1\overline{1}}\Big\}.
\end{eqnarray*}
It follows from \eqref{f-2}
\begin{eqnarray}\label{C2-9}
-\frac{1}{\chi_{1\overline{1}}}F^{i\overline{j},r\overline{s}}\chi_{i\overline{j};
1}\chi_{r\overline{s}; \overline{1}}&\geq&
\frac{1-\delta}{1+\delta}\frac{1}{\chi_{1\overline{1}}^2}\sum_{i \in I}F^{i\overline{i}}|\chi_{i\overline{1};
1}|^2\nonumber\\&\geq&\frac{1-\delta}{1+\delta}\frac{1}{\chi_{1\overline{1}}^2}\sum_{i \in I}F^{i\overline{i}}\Big(|\chi_{1\overline{1};
i}|^2+2\mathrm{Re}\{\chi_{1\overline{1};
i}\overline{b_i}\}\Big),
\end{eqnarray}
where $b_i=\chi_{0i \overline{1}; 1}-\chi_{0\overline{1}1; i}$.

Using \eqref{Diff-1}, we have in view of the fact $\varphi^{\prime \prime}=2(\varphi^{\prime})^2$
\begin{eqnarray}\label{C2-10}
&&\varphi^{\prime \prime}\sum_{i \in I}F^{i\overline{i}}|\nabla_i(|\nabla u|^2)|^2
\nonumber \\ &\geq& 2\sum_{i \in I}F^{i\overline{i}}\bigg(\delta \Big|\frac{\chi_{1\overline{1};
i}}{\chi_{1\overline{1}}}\Big|^2-\frac{\delta}{1-\delta}|\psi^{\prime}u_i|^2\bigg).
\end{eqnarray}
Choosing $\delta\leq\frac{1}{2A+1}$, we have $\psi^{\prime \prime}
\geq\frac{2\delta}{1-\delta}(\psi^{\prime})^2$. Then, we get by combining \eqref{C2-9} and \eqref{C2-10}
\begin{eqnarray}\label{C2-11}
&&-\sum_{i \in I}\frac{F^{i\overline{i}}|\chi_{1\overline{1};i}|^2}{\chi_{1\overline{1}}^2}
-(1-\delta^2)\frac{1}{\chi_{1\overline{1}}}F^{i\overline{j},r\overline{s}}\chi_{i\overline{j};
1}\chi_{r\overline{s}; \overline{1}}\nonumber\\&& \quad
+\varphi^{\prime\prime}\sum_{i \in I}F^{i\overline{i}}|\nabla_i(|\nabla u|^2)|^2
+\psi^{\prime\prime}\sum_{i \in I}F^{i\overline{i}}|u_i|^2\nonumber\\&\geq&\delta^2\sum_{i \in I}\frac{F^{i\overline{i}}|\chi_{1\overline{1};i}|^2}{\chi_{1\overline{1}}^2}
+\frac{2(1-\delta^2)(1-\delta)}{1+\delta}\frac{1}{\chi_{1\overline{1}}^2}\sum_{i \in I}F^{i\overline{i}}\mathrm{Re}\{\chi_{1\overline{1};
i}\overline{b_i}\}\nonumber\\&\geq&\delta^2\sum_{i \in I}\frac{F^{i\overline{i}}|\chi_{1\overline{1};i}|^2}{\chi_{1\overline{1}}^2}
-2(1-\delta)^2\frac{1}{\chi_{1\overline{1}}^2}\sum_{i \in I}F^{i\overline{i}}\mathrm{Re}\{\chi_{1\overline{1};
i}\overline{b_i}\}\nonumber\\&\geq&\frac{\delta^2}{2}\sum_{i \in I}\frac{F^{i\overline{i}}|\chi_{1\overline{1};i}|^2}{\chi_{1\overline{1}}^2}
-C\frac{4(1-\delta)^4}{\delta^2}\frac{1}{\chi_{1\overline{1}}^2}\sum_{i \in I}F^{i\overline{i}}
\nonumber\\ &\geq&\frac{\delta^2}{2}\sum_{i \in I}\frac{F^{i\overline{i}}|\chi_{1\overline{1};i}|^2}{\chi_{1\overline{1}}^2}
-C\sum_{i \in I}F^{i\overline{i}},
\end{eqnarray}
where we choose $\chi_{1\overline{1}}$ large enough to get the last inequality.

For the terms without an index in $I$, we have by \eqref{C2-8}
\begin{eqnarray}\label{C2-12}
&&-\sum_{i \notin I}\frac{F^{i\overline{i}}|\chi_{1\overline{1};i}|^2}{\chi_{1\overline{1}}^2}
+\varphi^{\prime\prime}\sum_{i \notin I}F^{i\overline{i}}|\nabla_i(|\nabla u|^2)|^2\nonumber\\&\geq&-8\Lambda^2K\sum_{i \notin I}F^{i\overline{i}}\nonumber\\&\geq&-\frac{8n\Lambda^2K}{\delta}F^{1\overline{1}},
\end{eqnarray}
where we used the fact $1 \notin I$ to get the last inequality.

Substituting \eqref{C2-11} and \eqref{C2-12} into \eqref{C2-4}
\begin{eqnarray}\label{C2-13}
&&C+C\sum_{i}F^{i\overline{i}}+\frac{8n\Lambda^2K}{\delta}F^{1\overline{1}}
\nonumber\\&\geq&\frac{1}{8K}\sum_{i}F^{i\overline{i}}\chi_{i\overline{i}}^2
+\psi^{\prime}\sum_{i}F^{i\overline{i}}u_{i\overline{i}}
.
\end{eqnarray}
Suppose that $\chi_{1\overline{1}}\geq N$. Since the cone condition \eqref{cone}
is equivalent to the condition \eqref{cone-1} for $\mu$ being the eigenvalues of $\chi_0$,
Lemma \ref{C2-212} works. Thus we will divide the following argument into
two cases:

\textbf{Case A: }If \eqref{C2-2-1} hold, we have
\begin{eqnarray*}
\sum_{i}F^{i\overline{i}} u_{i\overline{i}}
\leq-\theta-\theta \sum_{i}F^{i\overline{i}},
\end{eqnarray*}
which implies
\begin{eqnarray}\label{C2-14}
\psi^{\prime}\sum_{i}F^{i\overline{i}} u_{i\overline{i}}
\geq \Lambda\theta (1+\sum_{i}F^{i\overline{i}}).
\end{eqnarray}
Substituting \eqref{C2-14} into \eqref{C2-13}
\begin{eqnarray*}
&&C+C\sum_{i}F^{i\overline{i}}+\frac{8n\Lambda^2K}{\delta}F^{1\overline{1}}
\\&\geq&\frac{1}{8K}\sum_{i}F^{i\overline{i}}\chi_{i\overline{i}}^2
+\Lambda \theta (1+\sum_{i}F^{i\overline{i}}),
\end{eqnarray*}
which implies if we choose $\Lambda$ large enough
\begin{eqnarray*}
\frac{8n\Lambda^2K}{\delta}F^{1\overline{1}}
\geq\frac{1}{8K}\sum_{i}F^{i\overline{i}}\chi_{i\overline{i}}^2
\geq\frac{1}{8K}F^{1\overline{1}}\chi_{1\overline{1}}^2.
\end{eqnarray*}
Thus,
\begin{eqnarray*}
\chi_{1\overline{1}}^2\leq \frac{64n\Lambda^2K^2}{\delta}.
\end{eqnarray*}

\textbf{Case B:} Otherwise, we have by \eqref{C2-2-2}
\begin{eqnarray}\label{C200}
F^{1\overline{1}}\chi_{1\overline{1}}\geq \theta.
\end{eqnarray}
Thus, we get from \eqref{C2-6}
\begin{eqnarray}\label{C2-151}
&&C(1+\Lambda)+C\sum_{i}F^{i\overline{i}}+\frac{8n\Lambda^2K}{\delta}F^{1\overline{1}}
\nonumber\\&\geq&\frac{1}{8K}F^{1\overline{1}}\chi_{1\overline{1}}^2+\tau\Lambda\sum_{i}F^{i\overline{i}}
.
\end{eqnarray}
Choosing $\Lambda$ large enough, it follows from \eqref{C2-151}
\begin{eqnarray}\label{C2-16}
C(1+\Lambda)+\frac{8n\Lambda^2K}{\delta}F^{1\overline{1}}
\geq\frac{1}{8K}F^{1\overline{1}}\chi_{1\overline{1}}^2
.
\end{eqnarray}
Then, using \eqref{C200}, we choose $\chi_{1\overline{1}}$ large enough such that
\begin{eqnarray*}
\frac{1}{16K}F^{1\overline{1}}\chi_{1\overline{1}}^2\geq \frac{1}{16K}\theta \chi_{1\overline{1}}
\geq C(1+\Lambda).
\end{eqnarray*}
Taking the above inequality into \eqref{C2-16}, we arrive
\begin{eqnarray*}\label{C2-15}
\chi_{1\overline{1}}\leq \sqrt{\frac{128n}{\theta}}\Lambda K.
\end{eqnarray*}
So, we complete the proof.
\end{proof}

\section{The gradient estimate}

We will derive the gradient estimate using a blow-up argument
and Liouville-type theorem due to Dinew-Kolodziej \cite{Din12}. Suppose the gradient estimate fails. Then, there
exists a sequence solutions $u_m \in C^4(M)$ to the equation \eqref{K-eq2} such that
\begin{eqnarray*}
N_m=\sup_{M}|\nabla u_m|\rightarrow+\infty, \quad \mbox{as} \quad m\rightarrow+\infty.
\end{eqnarray*}
Clearly,
\begin{eqnarray}\label{m-1}
\chi^{k}_{u_m}\wedge \omega^{n-k}=\sum_{l=0}^{k-1}\alpha_l(z)\chi^{l}_{u_m}\wedge \omega^{n-l}, \quad \sup_{M}u_m=0.
\end{eqnarray}

For each $m$, we assume that $|\nabla u_m|$ attains its maximum value at $z_m \in M$. Then, after passing a subsequence,
$z_m$ converges some point $z \in M$. Choosing a normal coordinate chart around $z$, which we identify
with an open set in $\mathbb{C}^n$ with coordinates
$(z^1, ..., z^n)$, and such that $\omega(0)=\omega_0:=\frac{\sqrt{-1}}{2}\delta_{ij}dz^i\wedge d\overline{z}^j$.
Without loss of generality, we may assume that the open set contains $B_1(0)$. Set
\begin{eqnarray*}
v_m(z)=u\Big(\frac{z}{N_m}\Big), \quad z \in B_{N_m}(0).
\end{eqnarray*}
Then, we know from $C^2$ estimate
\begin{eqnarray*}
|v_m|_{C^2(B_{N_m}(0))}\leq C.
\end{eqnarray*}
By passing to a subsequence again, we can assume $v_m$ is $C^{1, \alpha}$ convergent
to a limit function  $v \in C^{1, \alpha}$ with $\nabla v(0)=1$. Then, we have from \eqref{m-1}
\begin{eqnarray*}\label{m-2}
\bigg[\chi\Big(\frac{z}{N_m}\Big)\bigg]^{k}\wedge \bigg[\omega\Big(\frac{z}{N_m}\Big)\bigg]^{n-k}
=\sum_{l=0}^{k-1}\alpha_l\Big(\frac{z}{N_m}\Big)
\bigg[\chi\Big(\frac{z}{N_m}\Big)\bigg]^{l}\wedge \bigg[\omega\Big(\frac{z}{N_m}\Big)\bigg]^{n-l},
\end{eqnarray*}
which results in
\begin{eqnarray*}\label{m-3}
&&\bigg[\chi_0\Big(\frac{z}{N_m}\Big)+N_{m}^{2}\frac{\sqrt{-1}}{2}\partial\overline{\partial} v_m\bigg]^{k}\wedge \bigg[\omega\Big(\frac{z}{N_m}\Big)\bigg]^{n-k}
\\&=&\sum_{l=0}^{k-1}\alpha_l\Big(\frac{z}{N_m}\Big)
\bigg[\chi_0\Big(\frac{z}{N_m}\Big)+N_{m}^{2}\frac{\sqrt{-1}}{2}\partial\overline{\partial} v_m\bigg]^{l}\wedge \bigg[\omega\Big(\frac{z}{N_m}\Big)\bigg]^{n-l}.
\end{eqnarray*}
Thus,
\begin{eqnarray*}\label{m-3}
&&N_{m}^{2k}\bigg[O\Big(\frac{1}{N_{m}^{2}}\Big)\omega_0+\frac{\sqrt{-1}}{2}\partial\overline{\partial}
v_m\bigg]^{k}\wedge \bigg[\omega_0+O\Big(\frac{|z|^2}{N^{2}_{m}}\Big)\omega_0\bigg]^{n-k}
\\&=&\sum_{l=0}^{k-1}\alpha_l\Big(\frac{z}{N_m}\Big)
N_{m}^{2l}\bigg[O\Big(\frac{1}{N_{m}^{2}}\Big)\omega_0+\frac{\sqrt{-1}}{2}\partial\overline{\partial}
v_m\bigg]^{l}\wedge \bigg[\omega_0+O\Big(\frac{|z|^2}{C^{2}_{m}}\Big)\omega_0\bigg]^{n-l}.
\end{eqnarray*}
Therefore, we have
\begin{eqnarray*}\label{m-3}
(\sqrt{-1}\partial\overline{\partial}
v)^{k}\wedge\omega_{0}^{n-k}=0,
\end{eqnarray*}
which is in the pluripotential sense.  Moreover, we have for any
$1\leq l\leq k$ by a similar reasoning
\begin{eqnarray*}
(\sqrt{-1}\partial\overline{\partial}
v)^{l}\wedge\omega_{0}^{n-l}\geq0.
\end{eqnarray*}
Thus, the limiting function $v$ is $k$-subharmonic.
A result of Blocki \cite{Blo05} tell us that $v$ is a maximal $k$-subharmonic function in $\mathbb{C}^n$.
Then the Liouville theorem in \cite{Din12} implies that $v$ is a constant, which contradicts
the fact $\nabla v(0)=1$.

\section{The proof of main theorem}

To get higher-order estimates, we need to show the equation \eqref{K-eq2} is uniformly elliptic.
First, we recall the following lemma.

\begin{lemma}
Under the assumption in Theorem \ref{Main}, we have
\begin{eqnarray}\label{k-1}
\sigma_{k-1}(\chi_{u})\geq C>0.
\end{eqnarray}
\end{lemma}

\begin{proof}
Without loss of generality, we assume there exists $0 \leq l_0 \leq k-2$ such that $\beta_{l_0}(z)>0$.
Then, we get from the inequality \eqref{2.9}
\begin{eqnarray*}
0<\frac{\sigma_{l_0}(\chi_u)}{\sigma_{k-1}(\chi_u)} \leq C.
\end{eqnarray*}
Thus, we complete the proof by using the generalized Newton-MacLaurin inequality \eqref{NM}
\end{proof}

Thus, using \eqref{k-1} and the uniform estimates up to second
order by Lemmas in the previous sections, we can deduce that
the equation \eqref{K-eq2} is uniformly elliptic.
Then, higher-order estimates follow from the Evans-Krylov theorem \cite{Eva82, Kry82}
and the Schauder estimate (see also \cite{Tos15}). Next, we shall
give a proof of Theorem \ref{Main} by the continuity method. The obstacle is that
we have to find a uniform cone condition for the solution flow
of the continuity method.

we define $\gamma(z)$  by
\begin{eqnarray*}
\chi_{0}^{k}\wedge \omega^{n-k}-
\sum_{l=1}^{k-2}\alpha_l\chi_{0}^{l}\wedge \omega^{n-l}=\gamma(z)\chi_{0}^{k-1}\wedge \omega^{n-k+1}
\end{eqnarray*}
It is easy to see from \eqref{initial} and \eqref{initial-1}
\begin{eqnarray}\label{cc1}
k\chi_{0}^{k-1}\wedge \omega^{n-k}>
\sum_{l=1}^{k-2}l\alpha_l\chi_{0}^{l-1}\wedge \omega^{n-l}+(k-1)\gamma\chi_{0}^{k-2}\wedge \omega^{n-k+1}.
\end{eqnarray}
Set
\begin{eqnarray*}
\widetilde{\alpha}_{k-1}(z)=\max\{\alpha_{k-1}(z), \gamma(z)\}.
\end{eqnarray*}
Then, $\chi_0$ satisfies the cone condition by combining \eqref{cc1} and \eqref{cone}
\begin{eqnarray}\label{cc2}
k\chi_{0}^{k-1}\wedge \omega^{n-k}&>&
\sum_{l=1}^{k-2}l\alpha_l\chi_{0}^{l-1}\wedge \omega^{n-l}\nonumber\\&&+(k-1)\widetilde{\alpha}_{k-1}\chi_{0}^{k-2}\wedge \omega^{n-k+1}.
\end{eqnarray}

We will apply the continuous method introduced in \cite{Tos19, Sun16, Sun17} to complete my proof.
First, we consider the family of the equations for $t \in [0, 1]$
\begin{eqnarray}\label{CM-1}
\chi_{u_t}^{k}\wedge \omega^{n-k}&=&\sum_{l=0}^{k-2}\alpha_l(z)\chi_{u_t}^{l}\wedge \omega^{n-l}
\nonumber\\&&+\Big[(1-t)\gamma(z)+t\widetilde{\alpha}_{k-1}(z)+a_t\Big]\chi_{u_t}^{k-1}\wedge \omega^{n-k+1},
\end{eqnarray}
where $\chi_{u_t} \in \Gamma_{k-1}(M)$ and $a_t$ is a constant for each $t$. We consider
\begin{eqnarray*}\label{CM-11}
\mathcal{T}_1=\{s \in [0, 1]: \exists \ u_t \in C^{2, \alpha}(M) \ \mbox{and} \ a_t \ \mbox{solving \eqref{CM-1} for} \ t \in [0, s] \}
\end{eqnarray*}
Clearly, $0 \in \mathcal{T}_1$ and $a_0=0$.

\begin{lemma}
$\mathcal{T}_1$ is open.
\end{lemma}

\begin{proof}
Rewriting the equations \eqref{CM-1} as the form
\begin{eqnarray*}\label{CM-1-open}
\log H(u_t, z, t):&=&\frac{\sigma_k(\chi_{u_t})}{\sigma_{k-1}(\chi_{u_t})}-\sum_{l=0}^{k-2}\beta_l
\frac{\sigma_l(\chi_{u_t})}{\sigma_{k-1}(\chi_{u_t})}\\&=&\frac{n-k+1}{k}\Big[(1-t)\gamma(z)+t\widetilde{\alpha}_{k-1}(z)+a_t\Big].
\end{eqnarray*}
In the following, we will write $H(u_t)=H(u_t, z, t)$ for short (suppressing the subscript $z, t$). Now we
follow the idea in \cite{Tos19, Sun16} to show $\mathcal{T}_1$ is open. Assume $\hat{t} \in \mathcal{T}_1$, we need to show that
there exists small $\varepsilon>0$ such that $t \in \mathcal{T}_1$ for any
$t \in [\hat{t}, \hat{t}+\varepsilon)$.

Define a new Hermitian metric corresponding to $u_t$
\begin{eqnarray*}
\Omega_t:=\frac{\sqrt{-1}}{2}H(u_t)H_{i\overline{j}}(u_t)dz^i\wedge dz^{\overline{j}}
\end{eqnarray*}
and
\begin{eqnarray*}
\widehat{\Omega}=\Omega_{\hat{t}}:=\frac{\sqrt{-1}}{2}H(u_{\hat{t}})H_{i\overline{j}}(u_{\hat{t}})dz^i\wedge dz^{\overline{j}},
\end{eqnarray*}
where $H^{i\overline{j}}=\frac{\partial H}{\partial u_{i\overline{j}}}$
and $\{H_{i\overline{j}}\}$ is the inverse of $H^{i\overline{j}}$.
Using Gauduchon's theorem \cite{Gau77} to $\widehat{\Omega}$, there exists a function $\hat{f}$ such that
$\widehat{\Omega}_{G}=e^{\hat{f}}\widehat{\Omega}$ is Gauduchon, i..e.,
$\partial\overline{\partial}(\widehat{\Omega}_{G}^{n-1})=0$. We may assume that by adding a constant to $\hat{f}$
\begin{eqnarray*}
\int_{M}e^{(n-1)\hat{f}}\widehat{\Omega}^{n}=1.
\end{eqnarray*}
We will show that we can find $u_t \in C^{2, \alpha}(M)$ for $t \in [\hat{t}, \hat{t}+\varepsilon)$ solving
\begin{eqnarray}\label{H-eq}
&&\frac{H(u_t)}{H(u_{\hat{t}})}\nonumber\\&=&\bigg(\int_{M}\frac{H(u_t)}{H(u_{\hat{t}})}e^{(n-1)\hat{f}}\widehat{\Omega}^{n}\bigg)
e^{\frac{n-k+1}{k}\Big[(\hat{t}-t)\gamma(z)+(t-\hat{t})\widetilde{\alpha}_{k-1}(z)+a_t-a_{\hat{t}}\Big]},
\end{eqnarray}
where $a_t$ is chosen such that
\begin{eqnarray*}
\int_{M}e^{\frac{n-k+1}{k}\Big[(\hat{t}-t)\gamma(z)+(t-\hat{t})
\widetilde{\alpha}_{k-1}(z)+a_t-a_{\hat{t}}\Big]}e^{(n-1)\hat{f}}\widehat{\Omega}^{n}=1.
\end{eqnarray*}
If $u_t$ is a solution to \eqref{H-eq}, $u_t+c$ is also for any $c\in \mathbb{R}$. Thus, we normalized $u_t$ such that
\begin{eqnarray*}
\int_{M}u_t e^{(n-1)\hat{f}}\widehat{\Omega}^{n}=0.
\end{eqnarray*}

Moreover, we define two Banach manifolds $\mathcal{B}_1$ and $\mathcal{B}_2$ by
\begin{eqnarray*}
\mathcal{B}_1:=\bigg\{\xi \in C^{2, \alpha}(M): \int_{M}\xi e^{(n-1)\hat{f}}\widehat{\Omega}^{n}=0\bigg\}
\end{eqnarray*}
and
\begin{eqnarray*}
\mathcal{B}_2:=\bigg\{\eta \in C^{\alpha}(M): \int_{M}e^{\eta} e^{(n-1)\hat{f}}\widehat{\Omega}^{n}=1\bigg\}.
\end{eqnarray*}
It is easy to see the tangent spaces of $\mathcal{B}_1$ and $\mathcal{B}_2$ at $0$:
\begin{eqnarray*}
T_0\mathcal{B}_1=\mathcal{B}_1 \quad
\mbox{and}
\quad T_0\mathcal{B}_2=\bigg\{\rho \in C^{\alpha}(M): \int_{M}\rho e^{(n-1)\hat{f}}\widehat{\Omega}^{n}=0\bigg\}.
\end{eqnarray*}
Define a linear operator $L: \mathcal{B}_1\rightarrow \mathcal{B}_2$ by
\begin{eqnarray*}
L(\xi):=\log H(u_{\hat{t}}+\xi)-\log H(u_{\hat{t}})-\log \bigg(\int_{M}\frac{H(u_{\hat{t}}+\xi)}{H(u_{\hat{t}})}e^{(n-1)\hat{f}}\widehat{\Omega}^{n}\bigg).
\end{eqnarray*}
Observe that $L(0)=0$.  By the inverse function theorem, to get openness of the
equation \eqref{H-eq} at $\xi=0$ we just need to show that
\begin{eqnarray*}
(DL)_0: T_0\mathcal{B}_1\rightarrow T_0\mathcal{B}_2
\end{eqnarray*}
is invertible. Compute
\begin{eqnarray*}
(DL)_0(\xi)&=&\frac{1}{H(u_{\hat{t}})}H^{i\overline{j}}(u_{\hat{t}})\xi_{ij}-
\int_{M}\frac{1}{H(u_{\hat{t}})}H^{i\overline{j}}(u_{\hat{t}})\xi_{ij} e^{(n-1)\hat{f}}\widehat{\Omega}^{n}
\\&=&\Delta_{\widehat{\Omega}}\xi-n\int_{M}e^{(n-1)\hat{f}}\widehat{\Omega}^{n-1}\wedge\partial\overline{\partial}\xi
\\&=&\Delta_{\widehat{\Omega}}\xi.
\end{eqnarray*}
where we are using the fact that $\partial\overline{\partial}(\widehat{\Omega}_{G}^{n-1})=0$.
Recall that we can solve the equation $\Delta_{\widehat{\Omega}_{G}}u=v$ as long as
$\int_{M}v\widehat{\Omega}_{G}^{n}=0$ (see e.g. \cite{Buc99}). Given $\rho \in T_0\mathcal{B}_2$ we see that
\begin{eqnarray*}
\int_{M}\rho e^{(n-1)\hat{f}}\widehat{\Omega}^{n}=\int_{M}\rho e^{-\hat{f}}\widehat{\Omega}_{G}^{n}=0.
\end{eqnarray*}
Then there exists $\eta$ solves
\begin{eqnarray*}
e^{-\hat{f}}\Delta_{\widehat{\Omega}}\eta=\Delta_{\widehat{\Omega}_{G}}\eta=\rho e^{-\hat{f}},
\end{eqnarray*}
which implies
\begin{eqnarray*}
\Delta_{\widehat{\Omega}}\eta=\rho,
\end{eqnarray*}
as required. This shows that we can find $u_t$ solving \eqref{H-eq} for
$t \in [\hat{t}, \hat{t}+\varepsilon)$.
Thus $\mathcal{T}_{1}$ is open.
\end{proof}

\begin{lemma}
$\mathcal{T}_1$ is closed.
\end{lemma}

\begin{proof}
It suffices to show (i) the uniform bound for $a_t$
and (ii) the uniform $C^{\infty}$ estimates for all $u_t$.
In fact, at the maximum point of $u_t$, using the concavity of
$(\frac{\sigma_k}{\sigma_l})^{\frac{1}{k-l}}$ (see Proposition \ref{Con}), we have
\begin{eqnarray*}
\frac{\chi_{0}^{k}\wedge \omega^{n-k}}{\chi_{0}^{l}\wedge \omega^{n-l}}
\geq\frac{\chi_{u_t}^{k}\wedge \omega^{n-k}}{\chi_{u_t}^{l}\wedge \omega^{n-l}}, \quad 0\leq l\leq k-1.
\end{eqnarray*}
Thus, we get from \eqref{cc1} and \eqref{cc2}
\begin{eqnarray*}
\chi_{0}^{k}\wedge \omega^{n-k}&\geq&\sum_{l=0}^{k-2}\alpha_l(z)\chi_{0}^{l}\wedge \omega^{n-l}
+\Big[(1-t)\gamma(z)+t\widetilde{\alpha}_{0}(z)+a_t\Big]\chi_{0}^{k-1}\wedge \omega^{n-k+1}
\\&\geq&\sum_{l=0}^{k-2}\alpha_l(z)\chi_{0}^{l}\wedge \omega^{n-l}
+\Big[\gamma(z)+a_t\Big]\chi_{0}^{k-1}\wedge \omega^{n-k+1}.
\end{eqnarray*}
Thus, $a_t\leq 0$. Similarly, at the minimum point of $u_t$, we have
\begin{eqnarray*}
\chi_{0}^{k}\wedge \omega^{n-k}&\leq&\sum_{l=0}^{k-2}\alpha_l(z)\chi_{0}^{l}\wedge \omega^{n-l}
+\Big[(1-t)\gamma(z)+t\widetilde{\alpha}_{0}(z)+a_t\Big]\chi_{0}^{k-1}\wedge \omega^{n-k+1}
\\&\leq&\sum_{l=1}^{k-1}\alpha_l(z)\chi_{0}^{l}\wedge \omega^{n-l}
+\Big[\widetilde{\alpha}_{0}(z)+a_t\Big]\chi_{0}^{k-1}\wedge \omega^{n-k+1}.
\end{eqnarray*}
Therefore, $a_t$ is uniformly bounded from below. Thus, the requirement (i) is met. For the requirement (ii), it is sufficient to show
the cone condition
\begin{eqnarray*}
k\chi_{0}^{k-1}\wedge \omega^{n-k}&>&\sum_{l=1}^{k-2}l\alpha_l\chi_{0}^{l-1}\wedge \omega^{n-l}
\\&&+(k-1)\Big[(1-t)\gamma(z)+t\widetilde{\alpha}_{0}(z)+a_t\Big]\chi_{0}^{k-2}\wedge \omega^{n-k+1}.
\end{eqnarray*}
is uniform for the equations \eqref{CM-1}.
This fact can be easily deduced in view of $a_t\leq 0$, \eqref{cc1} and \eqref{cc2}.
\end{proof}

Therefore, there exists a solution to the equation
\begin{eqnarray*}\label{CM-12}
\chi_{v}^{k}\wedge \omega^{n-k}&=&\sum_{l=0}^{k-2}\alpha_l(z)\chi_{v}^{l}\wedge \omega^{n-l}
+\Big[\widetilde{\alpha}_{k-1}(z)+a_1\Big]\chi_{0}^{k-1}\wedge \omega^{n-k+1}.
\end{eqnarray*}

Second, we consider the family of equations
\begin{eqnarray}\label{CM-2}
\chi_{u_t}^{k}\wedge \omega^{n-k}&=&\sum_{l=0}^{k-2}\alpha_l(z)\chi_{u_t}^{l}\wedge \omega^{n-l}
\\ \nonumber&&+\Big[(1-t)\widetilde{\alpha}_{k-1}(z)+t\alpha_{k-1}(z)+b_t\Big]\chi_{0}^{k-1}\wedge \omega^{n-k+1}.
\end{eqnarray}
where $\chi_{u_t} \in \Gamma_{k-1}(M)$ and $b_t$ is a constant for each $t$. We consider
\begin{eqnarray*}\label{CM-11}
\mathcal{T}_2=\{s \in [0, 1]: \exists \ u_t \in C^{2, \alpha}(M) \ \mbox{and} \ b_t \ \mbox{solving \eqref{CM-2} for} \ t \in [0, s] \}
\end{eqnarray*}
Clearly, $0 \in \mathcal{T}_2$ with $u_0=v$ and $b_0=a_1$. The openness of $\mathcal{T}_2$
and uniform bound for $b_t$ can be similarly deduced by the previous argument for \eqref{CM-1}.
Moreover, from the argument for $\mathcal{T}_1$, we know it suffices to show the cone condition holds for the equations \eqref{CM-2} in order to guarantee that the continuity method works.

Integrating \eqref{CM-2} on $M$, we get from the condition \eqref{Nece}
\begin{eqnarray*}
&&\int_{M}\chi_{0}^{k}\wedge \omega^{n-k}\\&=&\int_{M}\chi_{tu}^{k}\wedge \omega^{n-k}\\&=&\sum_{l=0}^{k-2} \int_{M}\alpha_l\chi_{tu}^{l}\wedge \omega^{n-l}+\int_{M}\Big[(1-t)\widetilde{\alpha}_{k-1}+t\alpha_{k-1}+b_t\Big]\chi_{0}^{k-1}\wedge \omega^{n-k+1}
\\&\geq&\sum_{l=0}^{k-2} c_{k,l}\int_{M}\chi_{0}^{l}\wedge \omega^{n-l}+(c_{k, k-1}+b_t)\int_{M}\chi_{0}^{k-1}\wedge \omega^{n-k+1}
\\&\geq&\int_{M}\chi_{0}^{k}\wedge \omega^{n-k}+b_t\int_{M}\chi_{0}^{k-1}\wedge \omega^{n-k+1}.
\end{eqnarray*}
Thus, $b_t\leq 0$. So the cone condition holds for the equations \eqref{CM-2}.
Therefore, we complete the proof of Theorem \ref{Main}.

%%%%%%%%%%%%%%%%%%%%%%%%%%%%%%%%%%%%%%%%%%%
 \bibliographystyle{siam}
% \bibliography{article}

\end{document}